\documentclass{article}
\usepackage[a4paper,margin=1in,footskip=0.25in]{geometry}
\usepackage{orcidlink}

\usepackage{graphicx}%
\usepackage{multirow}%
\usepackage{amsmath,amssymb,amsfonts}%
\usepackage{amsthm}%
\usepackage{mathrsfs}%
\usepackage{xcolor}%
\usepackage{textcomp}%
\usepackage{manyfoot}%
\usepackage{booktabs}%
\usepackage{algorithm, algpseudocode}
\usepackage{algorithmicx}%
\usepackage{algpseudocode}%
\usepackage{mathtools}
\makeatletter
\def\plist@algorithm{Alg.\space}
\makeatother

\usepackage{listings}%
\usepackage{cleveref}

\crefname{table}{tab.}{tabs.}
\Crefname{table}{Table}{Tables}

\usepackage{setspace}

\newcommand{\Matlab}{\textsc{Matlab}}

\DeclareMathOperator{\rank}{rank}

\usepackage{enumitem}

\usepackage{nicematrix}

\DeclareMathOperator*{\argmin}{argmin}
\DeclareMathOperator*{\argmax}{argmax}

\newcommand{\Stiefel}{\mathrm{St}}

\usepackage{xurl}

\renewcommand{\vec}{\mathrm{vec}}

\newcommand{\mlin}[1]{\mbox{\tt{#1}}}
\algnewcommand{\algorithmicor}{\textbf{ or }}
\algnewcommand{\Or}{\algorithmicor}
\algnewcommand{\And}{\algorithmicand}

\newtheorem{theorem}{Theorem}[section]
\newtheorem{lemma}[theorem]{Lemma}%
\newtheorem{proposition}[theorem]{Proposition}%
\newtheorem{definition}[theorem]{Definition}%

\usepackage{breakurl}
\usepackage{orcidlink}
\usepackage{xcolor}
\usepackage{circledsteps}
\usepackage{wrapfig}
\usepackage{fancyhdr}

\begin{document}

\title{A Riemannian rank-adaptive method for higher-order tensor completion in the tensor-train format}
\date{}
\author{
  Charlotte Vermeylen \and
  Marc Van Barel 
}

\maketitle

\begin{abstract}
In this paper a new Riemannian rank adaptive method (RRAM) is proposed for the low-rank tensor completion problem (LRTCP) formulated as a least-squares optimization problem on the algebraic variety of tensors of bounded tensor-train (TT) rank. The method iteratively optimizes over fixed-rank smooth manifolds using a Riemannian conjugate gradient algorithm from Steinlechner (2016) and gradually increases the rank by computing a descent direction in the tangent cone to the variety. Additionally, a numerical method to estimate the amount of rank increase is proposed based on a theoretical result for the stationary points of the low-rank tensor approximation problem and a definition of an estimated TT-rank. Furthermore, when the iterate comes close to a lower-rank set, the RRAM decreases the rank based on the TT-rounding algorithm from Oseledets (2011) and a definition of a numerical rank. We prove that the TT-rounding algorithm can be considered as an approximate projection onto the lower-rank set which satisfies a certain angle condition to ensure that the image is sufficiently close to that of an exact projection. Several numerical experiments are given to illustrate the use of the RRAM and its subroutines in {\Matlab}. Furthermore, in all experiments the proposed RRAM outperforms the state-of-the-art RRAM for tensor completion in the TT format from Steinlechner (2016) in terms of computation time.
\end{abstract}

\section{Introduction}\label{sec1}

We consider the \emph{low-rank tensor completion problem (LRTCP)} formulated as a least-squares optimization problem on the algebraic variety
$\mathbb{R}_{\le \left( k_1, \dots, k_{d-1} \right)}^{n_1 \times \dots \times n_d}$ of $n_1 \times \dots \times n_d$ real tensors of order $d$ and TT-rank at most $k:=\left( k_1, \dots, k_{d-1} \right)$ \cite[Definition 1.4]
{kutschan2018}:
\begin{equation} \label{eq:min_completion_TT}
\min_{X \in \mathbb{R}_{\le \left( k_1,\dots, k_{d-1} \right)}^{n_1 \times \cdots  \times n_d}} \underbrace{\frac{1}{2} \left\lVert X_{\Omega} - A_{\Omega} \right\rVert^2}_{=:f_\Omega(X)},
\end{equation}
where $A \in \mathbb{R}^{n_1 \times \cdots  \times n_d}$, $\Omega \subseteq \lbrace 1, \dots, n_1 \rbrace \times \cdots \times \lbrace 1, \dots, n_d \rbrace$ is called the sampling set,
\begin{equation} \label{eq:Z_Omega}
Z_{\Omega}\left(i_1,\dots,i_d\right) := \begin{cases}
Z\left(i_1,\dots,i_d \right)  & \text{if }\left(i_1,\dots,i_d \right) \in \Omega \\
0 & \text{otherwise} \\
\end{cases},
\end{equation}
for all $Z \in \mathbb{R}^{n_1 \times \cdots  \times n_d}$, and the norm is induced by the inner product\cite[Example~4.149]{Hackbusch}:
\begin{equation} \label{eq:inner_product}
\langle Y ,X \rangle = \langle \vec(Y) ,\vec(X) \rangle, \quad \forall~ X,Y \in \mathbb{R}^{n_1 \times \cdots  \times n_d}.
\end{equation}

A \emph{tensor-train decomposition (TTD)}\cite{Oseledets2011} of a tensor $X \in \mathbb{R}^{n_1 \times \dots \times n_d}$ is a factorization $X = X_1 \cdot X_2 \cdots X_{d-1} \cdot X_d$,
where $X_i \in \mathbb{R}^{r_{i-1} \times n_i \times r_i}$, and $r_0:=r_d:=1$. The `$\cdot$' indicates the multiplication of two tensors or a matrix and a tensor and more specifically the contraction between the last dimension of the first factor and the first dimension of the second factor: $X_{i} \cdot X_{i+1} := \left[ X_i^\mathrm{R} X_{i+1}^\mathrm{L} \right]^{r_{i-1} \times n_i \times n_{i+1} \times r_{i+1}}$,
where $X_i^\mathrm{R}$ and $X_{i+1}^\mathrm{L}$ are the right and left unfolding of a tensor respectively:
\begin{align*}
X_i^\mathrm{R} &:= \left[ X_i \right]^{r_{i-1} n_i \times r_i} := \mathrm{reshape}\left( X_i, r_{i-1} n_i \times  r_i \right), \\
X_{i+1}^\mathrm{L} &:= \left[ X_{i+1} \right]^{r_{i} \times n_{i+1} r_{i+1}}:= \mathrm{reshape}\left( X_{i+1}, r_{i} \times n_{i+1} r_{i+1} \right). 
\end{align*}
Remark that the left or right unfolding of a matrix is equal to the matrix itself.
An element in $X$ can thus be obtained as
\begin{equation*}
X(i_1, \dots, i_d) = \sum_{j_1=1}^{r_1} \cdots \sum_{j_{d-1}=1}^{r_{d-1}} X_1(i_1 ,j_1) X_2(j_1, i_2, j_2) \cdots X_d(j_{d-1}, i_d).
\end{equation*}
The minimal $r:=(r_1,\cdots, r_{d-1})$ for which a TTD of $X$ exists, is called the \emph{TT-rank} of $X$ or $\rank_{\mathrm{TT}} X$. For second-order tensors (matrices), the {TT-rank} reduces to the standard matrix rank. An advantage of the TTD is that the rank can be determined as the matrix rank of the unfoldings:
\begin{align} \label{eq:def_rank_TT}
r_i &:= \rank X^{<i>}, & X^{<i>} &:= \left[ X \right]^{n_1 \cdots n_{i} \times n_{i+1} \cdots n_d}.
\end{align}
A TTD of minimal rank can be obtained by computing successive SVDs of the unfoldings\cite[Algorithm 1]{Oseledets2011}. 
The set of TTDs of fixed rank $(r_1,\dots,r_{d-1})$ is known to be a \emph{smooth manifold} \cite[Lemma 4]{Holtz_manifolds_2012}:
\begin{equation} \label{eq:fixed_rank_manifold}
\mathbb{R}_{(r_1,\dots,r_{d-1})}^{n_1 \times \cdots \times n_d} := \left\lbrace X \in \mathbb{R}^{n_1 \times \cdots \times n_d} \mid \rank_{\mathrm{TT}} X = (r_1,\cdots, r_{d-1}) \right\rbrace,
\end{equation}
and the set of TTDs of bounded rank $k:=(k_1,\dots,k_{d-1})$ an algebraic variety \cite{kutschan2018}
\begin{equation} \label{eq:variety}
\mathbb{R}_{\le (k_1,\dots,k_{d-1})}^{n_1 \times \cdots \times n_d} := \left\lbrace X \in \mathbb{R}^{n_1 \times \cdots \times n_d} \mid \rank_{\mathrm{TT}} X \le (k_1,\dots,k_{d-1}) \right\rbrace,
\end{equation}
where the inequality applies element-wise.

Different Riemannian optimization methods on the smooth manifold have already been developed \cite{psenka2020second,cai2023tensor,manopt}. However, they require as input an adequate value for the TT-rank which is difficult to determine for most applications \cite{cai2023tensor,tensor_completioni_reg_TT_2020,Holtz2012}, and is therefore in general determined by trial-and-error. Futhermore, in practical LRTCPs, $A$ has usually full TT-rank due to noise. 



When $k$ is set too high however, the complexity of an algorithm to solve \eqref{eq:min_completion_TT} is unnecessarily high and furthermore overfitting can occur, i.e., $X$ approximates $A_\Omega$ well but not the full tensor $A$. To detect overfitting, usually a test data set $\Gamma$ is used \cite{Steinl_high_dim_TT_compl_2016}. When the error on this test set increases during optimization while the error of \eqref{eq:min_completion_TT} decreases overfitting has occurred and the algorithm should be stopped or the rank decreased. On the other hand, when $k$ is set too low, the search space may not contain a sufficiently good approximation of $A$\cite{Steinl_high_dim_TT_compl_2016,psenka2020second,TTcompletion_precond_kasai_2016,kressner2014low}. It is thus important to choose an adequate value for $k$.

The sampling ratio is defined as:
\begin{equation} \label{eq:rho_omega}
\rho_\Omega := \frac{\lvert \Omega \rvert}{n_1 \cdots n_d},
\end{equation}
where $\lvert \Omega \rvert$ denotes the number of elements in $\Omega$. The smaller $\rho_\Omega$, the more difficult it is to recover $A$ from $A_{\Omega}$ by solving \eqref{eq:min_completion_TT}. However, the minimal number of samples needed is not known \cite{budzinskiy2021tensor}.

In this paper, we propose a RRAM for higher-order tensor completion in the tensor-train format to resolve this difficulty. RRAMs are state-of-the-art methods that can be used to minimize a continuously differentiable function on a low-rank variety, a problem appearing, e.g., in low-rank matrix completion \cite{Zhou2016, gao2022riemannian}. The RRAMs in \cite{Zhou2016, gao2022riemannian} are developed for the set of bounded rank matrices and iteratively optimize over the smooth fixed-rank manifolds starting from a low initial rank. They increase the rank by performing a line search along a descent direction selected in the \emph{tangent cone} to the variety. This direction can be the projection of the negative gradient onto the tangent cone but does not need to; for instance, the RRAM developed by Gao and Absil \cite{gao2022riemannian} uses the projection of the negative gradient onto the part of the tangent cone that is normal to the tangent space. Additionally, they decrease the rank based on a truncated SVD when the fixed-rank algorithm converges to an element of a lower-rank set. This is possible because the manifold of fixed-rank matrices is not closed.

In this paper, we aim to generalize these RRAMs to the TT format. In this format, only the RRAM by Steinlechner is known to us from the literature \cite{Steinl_high_dim_TT_compl_2016}. This method is developed for high-dimensional tensor completion and has a random rank update mechanism in the sense that each TT-rank is increased subsequently by one by adding a small random term in the TTD of the current best approximation $X \in \mathbb{R}^{n_1 \times \cdots  \times n_d}_{(r_1, \dots, r_{d-1})}$:
\begin{align*}
X_{i}^\mathrm{R} &\gets \begin{bmatrix} X_{i}^\mathrm{R} & \varepsilon~ \texttt{randn} \left( r_{i-1}n_i \times 1\right) \end{bmatrix}, & X_{i+1}^\mathrm{L} \gets \begin{bmatrix} X_{i+1}^\mathrm{L} \\ \varepsilon~ \texttt{randn} \left( 1 \times n_{i+1} r_{i+1}  \right) \end{bmatrix},
\end{align*}
where $i \in \lbrace 1, \dots ,d-1 \rbrace$ and $\varepsilon$ is small, e.g., $10^{-8}$, such that $f_{\Omega}$ does not increase much, and \texttt{randn} is a built-in {\Matlab} function to generate normally distributed pseudo-random numbers.
The RRAM only terminates if a predefined maximal value for each $r_i$ is reached. 
Furthermore, no rank reduction step is included which makes the algorithm prone to overfitting. The full algorithm is available in the \textsc{Manopt} toolbox \cite{manopt}. 
 
We improve this RRAM by including a method to increase the TT-rank based on a descent direction in the tangent cone. The tangent cone is the set of all tangent vectors to the variety and is discussed in more detail in \Cref{sec:tangent_cone}. The full method to increase the rank is discussed in \Cref{sec:rank_incr_sec}. Furthermore, in \Cref{sec:rank_est_dim_d}, a numerical method is derived to determine how much the rank should be increased. Lastly, a method to decrease the rank is given in \Cref{sec:rank_red_gen}, which is necessary when the iterate comes close to a lower-rank set. This is possible because as for the manifold of fixed-rank matrices, the manifold of fixed-rank TTDs is not closed. This method can be considered as an \emph{approximate projection} to the lower-rank set. The approximate projection ensures that the image is sufficiently close to that of the true projection. Approximate projections are discussed in more detail in \Cref{sec:approx_proj}.

First, in \Cref{sec:prelim_RRAM} some preliminaries for the RRAM are given. Then, in \Cref{sec:rank_incr_sec} and \Cref{sec:rank_red_gen}, the methods to increase and decrease the rank are respectively proposed. Finally in \Cref{sec:RRAM_gen}, the full algorithm is given together with several numerical experiments to compare the proposed RRAM with the state-of-the-art RRAM \cite{Steinl_high_dim_TT_compl_2016}. 

\section{Preliminaries} \label{sec:prelim_RRAM}

In this section, we first give the general notation used in the rest of the paper and the preliminaries concerning the TTD. Afterwards, we introduce a compact notation to simplify the tensor expressions that are derived in the rest of the paper. Fourthly, the parametrization of the tangent cone is discussed \cite{kutschan2018}. We derive slightly different orthogonality conditions such that no matrix inverse is needed in the method to increase the rank, which improves the overall stability of the RRAM. Then, in \Cref{sec:LRTAP}, the auxiliary \emph{low-rank tensor approximation problem (LRTAP)} is defined.
Lastly, approximate projections are discussed. 

\subsection{Notation}

Tensors are denoted by capital letters. A matrix can be considered as a tensor of order two. Matrices are therefore also denoted by capital letters. On the other hand, scalars and vectors, which are zero- and one-dimensional `tensors' respectively, are denoted by lower-case letters. The order will always be clear from the context. 

\subsection{Properties tensor-train decomposition}

In this section, we review some properties of the TTD that are frequently used in the rest of the paper; we refer to the original paper \cite{Oseledets2011} and the subsequent works \cite{LR_tensor_methods_Steinlechner_2014,Steinl_high_dim_TT_compl_2016,Steinlechner_thesis2016} for more details. 

\begin{itemize}[leftmargin=*]
\item A TTD is not unique: 
$$X = \left( X_1 B_1 \right) \cdot \left( B_1^{-1} \cdot X_2 \cdot B_2 \right) \cdots \left( B_{d-2}^{-1} \cdot X_{d-1} \cdot B_{d-1} \right) \cdot \left( B_{d-1}^{-1} X_d \right),$$ 
where $B_i \in \mathrm{GL}(r_i)$, for $i=1, \dots, d-1$. Consequently, it can be proven that the dimension of the manifold in \eqref{eq:fixed_rank_manifold} is \cite{Holtz_manifolds_2012}
\begin{equation} \label{eq:dim_Rrn}
\mathrm{dim} \left( \mathbb{R}^{n_1 \times \cdots \times n_d}_{(r_1, \dots, r_{d-1})}\right) = \sum_{i=1}^d r_{i-1} n_i r_i - \sum_{i=1}^{d-1} r_i^2.
\end{equation}
Thus, similarly as for the SVD, \emph{orthogonality conditions} can be enforced to improve the numerical stability of algorithms working with TTDs \cite{Oseledets2011,Steinlechner_thesis2016,Steinl_high_dim_TT_compl_2016}.

\item To define orthogonality conditions, we let $\Stiefel(p, n) := \{U \in \mathbb{R}^{n \times p} \mid U^\top U = I_p\}$ denote the \emph{Stiefel manifold}, where $n, p \in \mathbb{N}$ with $n \ge p$. 

\item For every $U \in \Stiefel(p, n)$, we let $P_U := UU^\top$ and $P_U^\perp := I_n-P_U$
denote the orthogonal projections onto the range of $U$ and its orthogonal complement respectively.

\item A TTD is called \emph{i-orthogonal}, for $i \in \lbrace 1, \dots, d \rbrace$, if 
\begin{equation} \label{eq:i-orth}
X = X_1' \cdots X_{i-1}' \cdot \dot{X}_i \cdot X_{i+1}'' \cdots X_d'',
\end{equation}
where ${X_j'}^\mathrm{R} \in \Stiefel (r_j,r_{j-1}n_j)$, for $j=1,\dots, i-1$, and $\left({X_k''}^\mathrm{L}\right)^\top \in \Stiefel (r_{k-1},n_k r_{k})$, for $k=i+1,\dots, d$, and $\dot{X}_i \in \mathbb{R}^{r_{i-1} \times n_i \times r_i}$. 
The tensors $X_i''$ are called \emph{left-orthogonal} and the tensors $X_i'$ \emph{right-orthogonal}.
It holds that
\begin{align}  \label{eq:dotXifvXiQ}
\dot{X}_i^\mathrm{R} &= {X_i'}^\mathrm{R} Q_i, & \dot{X}_i^{\mathrm{L}} &= R_i {X_i''}^{\mathrm{L}},
\end{align}
for some $Q_i \in \mathrm{GL}(r_i)$ and $R_i \in \mathrm{GL}(r_{i-1})$. Thus, 
\begin{align} \label{eq:def_Qi_Ri_orth_X_dim_d}
Q_i &= {X_i'}^{\mathrm{R},\top} {\dot{X}_i}^{\mathrm{R}}, & R_i =\dot{X}_i^{\mathrm{L}} {X_i''}^{\mathrm{L},\top}.
\end{align}
and $X$ can be written as
\begin{equation} \label{eq:X_Qi}
\begin{split}
X &= \dot{X}_1 R_2^{-1} \cdot \dot{X}_2 \cdot R_3^{-1} \cdot \dot{X}_3 \cdots R_d^{-1} \dot{X}_d = \dot{X}_1 Q_1^{-1} \cdot \dot{X}_2 \cdot Q_2^{-1} \cdot \dot{X}_3 \cdots Q_{d-1}^{-1} \dot{X}_d, 
\end{split}
\end{equation}
and consequently $Q_i = R_{i+1}$, for $i=1 ,\dots, d-1$. We can derive a similar expression for $A \in \mathbb{R}_{(r_1', \dots, r_{d-1}')}^{n_1 \times \dots \times n_d}$:
\begin{equation} \label{eq:A_Qi}
\begin{split}
A &= \dot{A}_1 {Q_1'}^{-1} \cdot \dot{A}_2 \cdot {Q_2'}^{-1} \cdot \dot{A}_3 \cdots {Q_{d-1}'}^{-1} \dot{A}_d, 
\end{split}
\end{equation}
where $Q_i' = {A_i'}^{\mathrm{R},\top} {\dot{A}_i}^{\mathrm{R}}=\dot{A}_{i+1}^{\mathrm{L}} {A_{i+1}''}^{\mathrm{L},\top}$, and where the factors $\dot{A}_i$, $A_i'$, and $A_i''$ are defined in the same way as in \eqref{eq:i-orth}.

\item The unfoldings $X^{<i>}$, defined in \eqref{eq:def_rank_TT}, of a TTD $X=X_1 \cdots X_d$ can be rewritten as:
\begin{equation} \label{eq:LR_TTD}
\begin{split}
&X^{<i>} = \left( X_1 \cdots X_{i} \right)^\mathrm{R} \left( X_{i+1} \cdots X_{d} \right)^\mathrm{L}.
\end{split}
\end{equation}
From \eqref{eq:def_rank_TT} and \eqref{eq:LR_TTD} it can be deduced that
\begin{equation} \label{eq:rank_factors_TTD1}
\begin{split}
r_i &= \rank \left( X_1 \cdots X_{i} \right)^\mathrm{R} = \rank \left( X_{i+1} \cdots X_{d} \right)^\mathrm{L},
\end{split}
\end{equation}
which can be used to prove that the left and right unfolding of each factor $X_i$ has full rank $r_{i-1}$ and $r_i$ respectively.

\item Orthogonality between TTDs is exploited frequently in the rest of this paper, e.g., in the parametrization of the tangent cone discussed in \Cref{sec:tangent_cone}. Orthogonality is defined with respect to the inner product \eqref{eq:inner_product}.
Remark that the vectorization of an unfolding is equal to that of the tensor. Thus, if $Y = Y_1 \cdots Y_d$ and $Z = Z_1 \cdots Z_d$, and by using \eqref{eq:LR_TTD}, the inner product $\langle Y,Z \rangle$ is zero if at least one of the following equalities holds:
\begin{align} \label{eq:inner_prod_zero}
\left( \left( Y_1 \cdots Y_{i} \right)^\mathrm{R} \right)^\top \left( Z_1 \cdots Z_i \right)^\mathrm{R} &= 0, &\left( Y_{i+1} \cdots Y_d \right)^\mathrm{L} \left( \left( Z_{i+1} \cdots Z_d \right)^\mathrm{L} \right)^\top &= 0,
\end{align}
for $i=1, \dots, d-1$. 
\end{itemize}

\subsection{Compact notations} \label{sec:notation}
To simplify the expressions in the rest of the paper we introduce a compact notation for the reshape operator of a higher-order tensor to a third-order tensor:
\begin{align*}
Y^{<i,j>} &:= \left[ Y \right]^{n_1 \cdots n_{i} \times n_{i+1} \cdots n_j \times n_{j+1} \cdots n_d} := \mathrm{reshape}\left(Y,n_1 \cdots n_{i} \times n_{i+1} \cdots n_j \times n_{j+1} \cdots n_d \right),
\end{align*}
where $1 \leq i \leq j \leq d$. Remark that if $i=j$ or $j=d$ this notation is equal to $Y^{<i>}$ defined in \eqref{eq:def_rank_TT}.

Additionally, we let $X^{\mathrm{L},\top} := \left( X^{\mathrm{L}}\right)^{\top}$, and similarly for $X^{\mathrm{R},\top}$. Lastly, we define
\begin{align*}
X_{i:j} &:= X_i \cdots X_j \in \mathbb{R}^{r_{i-1} \times n_i \times \cdots \times n_j \times r_j}, &  1 \leq i &< j \leq d,
\end{align*}
Similarly,
\begin{align*}
X'_{i:j} &:= X_i' \cdots X_j' \in \mathbb{R}^{r_{i-1} \times n_i \times \cdots \times n_j \times r_j}, & 1 \leq i &< j \leq d-1, \\ 
X''_{i:j} &:= X_i'' \cdots X_j'' \in \mathbb{R}^{r_{i-1} \times n_i \times \cdots \times n_j \times r_j}, & 2 \leq i &< j \leq d.
\end{align*}
Remark that $\left( X''_{i:j} \right)^{\mathrm{L},\top} \in \Stiefel \left(r_{i-1},n_i \cdots n_j r_j \right)$ and $\left( X'_{i:j} \right)^{\mathrm{R}} \in \Stiefel \left(r_{j},r_{i-1} n_i \cdots n_j \right)$ for any $i$ and $j$ that satisfy the above bounds.

\subsection{Tangent cone}\label{sec:tangent_cone}
In the next lemma, we recall the parametrization of the tangent cone to the variety of TTDs of bounded TT-rank \cite{kutschan2018} with modified orthogonality conditions.

\begin{lemma}
\label{lemma:TangentConeLowRankVariety_dim_d}
Let $X \in \mathbb{R}_{(r_1, \dots, r_{d-1})}^{n_1 \times \cdots  \times n_d}$ as in \eqref{eq:i-orth}.
Then, $T_X \mathbb{R}_{\le (k_1, \dots, k_{d-1})}^{n_1 \times \cdots \times n_d}$ is the set of all tensors $G \in \mathbb{R}^{n_1 \times \cdots  \times n_d}$ that can be decomposed as
\begin{equation}
\label{eq:param_g_dim_d}
G = \begin{bmatrix} X_1'& U_1 & W_1 \end{bmatrix} \cdot
\begin{bmatrix}
X_2' & U_2 & W_2 \\ 
0 & Z_2 & V_2 \\
0 & 0 & X_2'' \\
\end{bmatrix} \cdots \begin{bmatrix}
X_{d-1}' & U_{d-1} & W_{d-1} \\ 
0 & Z_{d-1} & V_{d-1} \\
0 & 0 & X_{d-1}'' \\
\end{bmatrix} \cdot \begin{bmatrix} W_d \\ V_d \\ X_d'' \end{bmatrix},
\end{equation}
where $U_i \in \mathbb{R}^{r_{i-1} \times n_i \times s_i}$, $s_i := k_i-r_i$, for $i=1,\dots,{d-1}$, $W_i \in \mathbb{R}^{r_{i-1} \times n_i \times r_i}$, for $i=1,\dots,{d}$, $Z_i \in \mathbb{R}^{s_{i-1} \times n_i \times s_i}$, for $i=2,\dots,{d-1}$, $V_i \in \mathbb{R}^{s_{i-1} \times n_i \times r_i}$, for $i=2,\dots,{d}$, and
\begin{align}
\label{eq:orth_cond_mod_dim_d}
(U_i^{\mathrm{R}})^\top X_i'^{\mathrm{R}} &= 0, & (W_i^{\mathrm{R}})^\top X_i'^{\mathrm{R}} &= 0, & i&=1,\dots, d-1, \nonumber \\
V_i{^\mathrm{L}} (X_i''^{\mathrm{L}})^\top &= 0, &&& i&=2 ,\dots, d.
\end{align}
\end{lemma}
\begin{proof}
$T_X \mathbb{R}_{\le (k_1, \dots, k_{d-1})}^{n_1 \times \cdots \times n_d}$ is the set of all $G \in \mathbb{R}^{n_1 \times \cdots \times n_d}$ that can be decomposed as \cite[Theorem 2.6]{kutschan2018}
\begin{equation*}
\begin{split}
G &= \begin{bmatrix} X_1' & {U}_1 & \dot{W}_1 \end{bmatrix} \cdot
\begin{bmatrix}
X_2' & {U}_2 & \dot{W}_2 \\ 
0 & Z_2 & \dot{V}_2 \\
0 & 0 & X_2' \\
\end{bmatrix} \cdots \begin{bmatrix}
X_{d-1}' & {U}_{d-1} & \dot{W}_{d-1} \\ 
0 & {Z}_{d-1} & \dot{V}_{d-1} \\
0 & 0 & X_{d-1}' \\
\end{bmatrix} \cdot \begin{bmatrix} {W}_d \\ {V}_d \\ \dot{X}_d \end{bmatrix}, \\
\end{split}
\end{equation*}
with the orthogonality conditions \cite[Theorem 2.6]{kutschan2018}
\begin{align} \label{eq:orth_original_dim_d}
\left( X_i'^{\mathrm{R}} \right)^\top {U}_i^\mathrm{R}&= 0, & \left( X_i'^{\mathrm{R}} \right)^\top \dot{W}_i^\mathrm{R} &=0, & i&=1,\dots, d-1, \nonumber \\
\left( \dot{V}_i \cdot X_{(i+1):(d-1)}' \cdot \dot{X}_d \right)^\mathrm{L} \left( \left( X_{i:(d-1)}' \cdot \dot{X}_d \right)^\mathrm{L} \right)^\top &= 0, & & & i &= 2,\dots, d.
\end{align} 
The following invariances hold for all invertible matrices $B_i$ in $\mathrm{GL}(r_{i})$, for $i=1,\dots, d-1$,
\begin{equation*}
\begin{split}
G =\begin{bmatrix} X_1' & {U}_1 & \dot{W}_1 B_1 \end{bmatrix} \cdot
&\begin{bmatrix}
X_2' & {U}_2 & \dot{W}_2 \cdot B_2 \\ 
0 & {Z}_2 & \dot{V}_2 \cdot B_2 \\
0 & 0 & B_1^{-1} \cdot X_2' \cdot B_2 \\
\end{bmatrix} \cdots \\
&\begin{bmatrix} X_{d-1}' & {U}_{d-1} & \dot{W}_{d-1} \cdot B_{d-1} \\ 
0 & {Z}_{d-1} & \dot{V}_{d-1} \cdot B_{d-1} \\
0 & 0 & B_{d-2}^{-1} \cdot X_{d-1}' \cdot B_{d-1} \\
\end{bmatrix} \cdot \begin{bmatrix} {W}_d \\ {V}_d \\ B_{d-1}^{-1} \dot{X}_d \end{bmatrix}. 
\end{split}
\end{equation*}
Because this holds for all invertible $B_i$, it also holds for the matrices $Q_i$ defined in \eqref{eq:def_Qi_Ri_orth_X_dim_d}, and thus because of \eqref{eq:X_Qi}, $X_i'' := Q_{i-1}^{-1} \cdot X_i'\cdot Q_{i}$, for $i=2,\dots, d$, where $Q_d:=I_{n_d}$. Then, we define $W_i := \dot{W}_i Q_{i}$ and $V_i := \dot{V}_i Q_{i}$, for $i=1,\dots, d-1$, and obtain the parametrization \eqref{eq:param_g_dim_d}.
The second set of orthogonality conditions in \eqref{eq:orth_original_dim_d} changes to
\begin{align*}
\left( X_i'^{\mathrm{R}} \right)^\top \dot{W}_i^\mathrm{R} &= 0 \Leftrightarrow \left( X_i'^{\mathrm{R}} \right)^\top \dot{W}_i^\mathrm{R} Q_{i} =0 \Leftrightarrow \left( X_i'^{\mathrm{R}} \right)^\top {W}_i^\mathrm{R} = 0, & i&=1,\dots, d,
\end{align*}
and the third changes to
\begin{align*}
\left( \dot{V}_i \cdot X_{(i+1):(d-1)}' \cdot \dot{X}_d \right)^\mathrm{L} \left( \left( X_{i:(d-1)}' \cdot \dot{X}_d \right)^\mathrm{L} \right)^\top = 0 \Leftrightarrow \left( \dot{V}_i \cdot Q_i \cdot X_{(i+1):d}''\right)^\mathrm{L} \left( \left( Q_{i-1} \cdot X_{i:d}'' \right)^\mathrm{L} \right)^\top &= 0 \\
\Leftrightarrow \left( {V}_i \cdot X_{(i+1):d}''\right)^\mathrm{L} \left( \left(  X_{i:d}''\right)^\mathrm{L} \right)^\top Q_{i-1}^\top &= 0 \\
\Leftrightarrow {V}_i^\mathrm{L} \left( \left( X_{(i+1):d}''\right)^\mathrm{L} \otimes I_{n_i} \right) \left( \left( X_{(i+1):d}''\right)^{\mathrm{L},\top} \otimes I_{n_i} \right) \left(X_{i}''\right)^{\mathrm{L},\top} &= 0 \\
\Leftrightarrow {V}_i^\mathrm{L} \left(X_{i}''\right)^{\mathrm{L},\top} &= 0,
\end{align*} 
for $i = 2,\dots, d$. The first set of orthogonality conditions in \eqref{eq:orth_original_dim_d} is left unchanged.
\end{proof}

Thus, when \eqref{eq:param_g_dim_d} is expanded, a sum of $\frac{d(d+1)}{2}$ mutual orthogonal terms is obtained: 
\begin{align} \label{eq:G_expanded_d}
G = X_{1:(d-1)}' \cdot W_d ~+ &\cdots + W_1 \cdot X_{2:d}''~ + \nonumber \\
X_{1:(d-2)}' \cdot U_{d-1} \cdot V_d ~ + &\cdots + 
U_1 \cdot V_2 \cdot X_{3:d}''~+ \nonumber\\
 X_{1:(d-3)}' \cdot U_{d-2} \cdot Z_{d-1} \cdot V_d~ + &\cdots + U_1 \cdot Z_2 \cdot V_3 \cdot X_{4:d}''~+ \\
&~~ \vdots \nonumber \\
+~ U_1 \cdot Z_2 &\cdots Z_{d-1} \cdot V_d. \nonumber
\end{align}
The tangent space is then defined as
\begin{equation} \label{eq:tangent_space_dim_d}
\begin{split}
T_X\mathbb{R}_{(r_1, \dots r_{d-1})}^{n_1 \times \cdots \times n_d} := \Big\lbrace 
~&X_{1:(d-1)}' \cdot W_d ~+ \cdots ~ +~ W_1 \cdot X_{2:d}'' \Bigm\vert \mathrm{s.t. }~ W_i^{\mathrm{R},\top} {X_i'}^\mathrm{R} = 0,~ i=1, \dots, d-1~  \Big\rbrace.  \\
\end{split}
\end{equation}
Remark that the orthogonality conditions of the tangent space can easily be changed to 
\begin{align} \label{eq:orth_cond_tangent_space_dim_d}
W_i^{\mathrm{R},\top} {X_i'}^\mathrm{R} &= 0, & i&=1, \dots, j-1, \\
W_i^{\mathrm{L}} {X_i''}^\mathrm{L,\top}&= 0, &  i&=j+1, \dots, d, \nonumber
\end{align}
for any $j \in \lbrace 1,\dots,d \rbrace$. For example, we can decompose $W_d$ in \eqref{eq:tangent_space_dim_d} as $W_d = \dot{W}_d X_d'' + \hat{W}_d$, where $\hat{W}_d := W_d \left(I - {X_d''}^\top X_d''\right) $ and $\dot{W}_d := W_d {X_d''}^\top$. Consequently $X_d''\hat{W_d}^\top =0$, and we can regroup the terms involving $W_{d-1}$ and $W_{d}$ as:
\begin{equation*}
\begin{split}
&X_{1:(d-2)}' \cdot W_{d-1} \cdot X_d'' + X_{1:(d-1)}' \cdot W_{d} = X_{1:(d-2)}' \cdot \left( W_{d-1} + X_{d-1}' \cdot \dot{W}_d \right) \cdot X_d''+ X_{1:(d-1)}' \cdot \hat{W}_d .
\end{split}
\end{equation*}
And thus, if we define $\tilde{W}_{d-1}:= W_{d-1} + X_{d-1}' \cdot \dot{W}_d$, then the modified parameterization of the tangent space with parameters $W_1, \dots, W_{d-2},\tilde{W}_{d-1},\hat{W}_d$ satisfies \eqref{eq:orth_cond_tangent_space_dim_d} for $j=d-1$. This process can be applied recursively to obtain \eqref{eq:orth_cond_tangent_space_dim_d} for any $j$.

As for third-order tensors, the projection onto the tangent space is easy and well known \cite{Steinl_high_dim_TT_compl_2016,Lubich_TT_time_int_2015}. For the parametrization in \Cref{lemma:TangentConeLowRankVariety_dim_d}, the parameters $W_i$, $i=1,\dots, d$, of $\mathcal{P}_{T_X\mathbb{R}_{(r_1, \dots, r_{d-1})}^{n_1 \times \cdots \times n_d}} Y$ with $Y \in \mathbb{R}^{n_1 \times \cdots n_d}$ are:
\begin{align} \label{eq:param_tan_space_dim_d}
W_i^\mathrm{R} &= P_{{X_i'}^\mathrm{R}}^\perp \left(\left( X_{1:(i-1)}' \right)^{\mathrm{R},\top} \cdot Y ^{<i >} \cdot \left( X_{(i+1):d}'' \right)^{\mathrm{L},\top} \right)^\mathrm{R}, & i&=1,\dots, d-1 \nonumber \\
W_d &= \left( X_{1:(d-1)}'\right)^{\mathrm{R},\top} Y^\mathrm{R}. 
\end{align}
Or using the orthogonality conditions \eqref{eq:orth_cond_tangent_space_dim_d}:
\begin{align} \label{eq:param_tan_space_dim_d_gen}
W_i^\mathrm{R} &= P_{{X_i'}^\mathrm{R}}^\perp \left(\left( X_{1:(i-1)}' \right)^{\mathrm{R},\top} \cdot Y ^{<i >} \cdot \left( X_{(i+1):d}'' \right)^{\mathrm{L},\top} \right)^\mathrm{R}, & i&=1,\dots, j-1, \nonumber \\
W_j &= \left( X_{1:(j-1)}' \right)^{\mathrm{R},\top} \cdot Y ^{<j >} \cdot \left( X_{(j+1):d}'' \right)^{\mathrm{L},\top}, \\
W_i^\mathrm{L} &=  \left(\left( X_{1:(i-1)}' \right)^{\mathrm{R},\top} \cdot Y ^{<i >} \cdot \Big( X_{(i+1):d}'' \Big)^{\mathrm{L},\top} \right)^\mathrm{L} P_{{X_i''}^{\mathrm{L},\top}}^\perp, & i&=j+1,\dots, d \nonumber.
\end{align}

\subsection{Low-rank tensor approximation problem} \label{sec:LRTAP}

We define the auxiliary \emph{low-rank tensor approximation problem (LRTAP)} as:
\begin{align} \label{eq:low-rank_approx}
&\min_{X \in \mathbb{R}_{\le (r_1,\dots, r_{d-1})}^{n_1 \times \cdots  \times n_d}} \underbrace{ \frac{1}{2} \lVert X - A  \rVert^2}_{=: f(X)}, & &A \in \mathbb{R}^{n_1 \times \cdots  \times n_d},
\end{align}
This problem is related to the LRTCP \eqref{eq:min_completion_TT} because $f_\Omega(X) = f(X)$ for $\Omega = \lbrace 1, \dots, n_1 \rbrace \times \cdots \times \lbrace 1, \dots, n_d \rbrace$. Remark that, as for \eqref{eq:min_completion_TT}, a global minimizer is, in general, not unique because $\mathbb{R}_{\le (k_1,\dots, k_{d-1})}^{n_1 \times \cdots  \times n_d}$ is non-convex and NP-hard to obtain \cite{hillar2013most}. This problem is used in \Cref{sec:rank_est_dim_d,sec:rank_red_gen}.


\subsection{Approximate projection} \label{sec:approx_proj}
To develop a RRAM, we need to reduce the rank when the iterate comes close to a lower-rank set. Ideally, we want to solve:
\begin{align} \label{eq:proj_variety_general}
&\mathcal{P}_{\mathbb{R}^{n_1 \times \cdots  \times n_d}_{\leq (r_1,\dots, r_{d-1})}}Y := \argmin_{X \in \mathbb{R}^{n_1 \times \cdots  \times n_d}_{\leq (r_1,\dots, r_{d-1})}} \lVert Y - X \rVert^2, & Y &\in \mathbb{R}^{n_1 \times \cdots  \times n_d}_{(r_1',\dots, r_{d-1}')},
\end{align}
where $\lVert \cdot \rVert$ is again the norm induced by \eqref{eq:inner_product} and $r_{i} \leq r_i'$, for $i=1,\dots,d-1$. Since the low-rank variety is a closed cone, it holds that \cite[Proposition~A.6]{LevinKileelBoumal2022}
\begin{align} \label{eq:nec_cond_approx_proj}
\big\langle Y-\hat{Y},\hat{Y} \big\rangle=0 \Leftrightarrow \big\langle Y,\hat{Y} \big\rangle = \big\lVert \hat{Y} \big\rVert^2, 
\end{align}
for all $Y \in \mathbb{R}^{n_1 \times \cdots  \times n_d}_{(r_1',\dots, r_{d-1}')}$ and $\hat{Y} \in \mathcal{P}_{\mathbb{R}^{n_1 \times \cdots  \times n_d}_{\leq (r_1,\dots, r_{d-1})}} Y$.
Thus, \eqref{eq:proj_variety_general} can be rewritten as
\begin{equation*} 
\begin{split}
\mathcal{P}_{\mathbb{R}^{n_1 \times \cdots  \times n_d}_{\leq (r_1,\dots, r_{d-1})}} Y
= \argmin_{\substack{X \in \mathbb{R}^{n_1 \times \cdots  \times n_d}_{\leq (r_1,\dots, r_{d-1})} \\ \langle Y, X \rangle = \|X\|^2}} \|Y-X\|^2 &= \argmin_{\substack{X \in \mathbb{R}^{n_1 \times \cdots  \times n_d}_{\leq (r_1,\dots, r_{d-1})} \\ \langle Y, X \rangle = \|X\|^2}} - 2 \langle Y, X \rangle + \|X\|^2 \\
&= \argmax_{\substack{X \in \mathbb{R}^{n_1 \times \cdots  \times n_d}_{\leq (r_1,\dots, r_{d-1})} \\ \langle Y, X \rangle = \|X\|^2}} \|X\| \\
&= \argmax_{\substack{X \in \mathbb{R}^{n_1 \times \cdots  \times n_d}_{\leq (r_1,\dots, r_{d-1})} \\ \left\langle Y, X \right\rangle = \|X\|^2}} \left\langle Y, \frac{X}{\|X\|} \right\rangle,
\end{split}
\end{equation*}
or
\begin{align}
\label{eq:proj_max_reform_gen2}
\big\lVert \hat{Y} \big\rVert = \max_{\substack{X \in \mathbb{R}^{n_1 \times \cdots  \times n_d}_{\leq (r_1,\dots, r_{d-1})} \\ \left\langle Y, X \right\rangle = \|X\|^2}} \left\langle Y, \frac{X}{\|X\|} \right\rangle = \max_{\substack{X \in \mathbb{R}^{n_1 \times \cdots  \times n_d}_{\leq (r_1,\dots, r_{d-1})} \\ \left\langle Y, X \right\rangle = \|X\|^2}} \|X\|,
\end{align}
for all $\hat{Y} \in \mathcal{P}_{\mathbb{R}^{n_1 \times \cdots  \times n_d}_{\leq (r_1,\dots, r_{d-1})}} Y$, and thus all elements of $\mathcal{P}_{\mathbb{R}^{n_1 \times \cdots  \times n_d}_{\leq (r_1,\dots, r_{d-1})}} Y$ have the same norm. However, as both varieties are highly non-convex, solving \eqref{eq:proj_max_reform_gen2} is non-trivial. That is why we search for an \emph{approximate projection}, i.e., a set-valued mapping:
$$\tilde{\mathcal{P}}_{\mathbb{R}^{n_1 \times \cdots  \times n_d}_{\leq (r_1,\dots, r_{d-1})}} : \mathbb{R}^{n_1 \times \cdots  \times n_d}_{(r_1',\dots, r_{d-1}')} \multimap \mathbb{R}^{n_1 \times \cdots  \times n_d}_{\leq (r_1,\dots, r_{d-1})},$$ 
such that there exists $\omega \in (0, 1]$ such that, for all $Y \in \mathbb{R}^{n_1 \times \cdots  \times n_d}_{(r_1',\dots, r_{d-1}')}$ and all $\tilde{Y} \in \tilde{\mathcal{P}}_{\mathbb{R}^{n_1 \times \cdots  \times n_d}_{\leq (r_1,\dots, r_{d-1})}} Y$,
\begin{equation}
\label{eq:ApproximateProjectionAngleCondition}
\left\langle Y, \frac{\tilde{Y}}{\big\lVert \tilde{Y} \big\rVert} \right\rangle \ge \omega \Big\lVert \mathcal{P}_{\mathbb{R}^{n_1 \times \cdots  \times n_d}_{\leq (r_1,\dots, r_{d-1})}}Y \Big\rVert .
\end{equation}
Inequality~\eqref{eq:ApproximateProjectionAngleCondition} is called an \emph{angle condition} \cite[Definition~2.5]{schneider2015convergence}; it is well defined since, as $\mathbb{R}^{n_1 \times \cdots  \times n_d}_{\leq (r_1,\dots, r_{d-1})}$ is a closed cone, all elements of $\mathcal{P}_{\mathbb{R}^{n_1 \times \cdots  \times n_d}_{\leq (r_1,\dots, r_{d-1})}} Y$ have the same norm. 
If the approximate projection is chosen such that all $\tilde{Y} \in \tilde{\mathcal{P}}_{\mathbb{R}^{n_1 \times \cdots  \times n_d}_{\leq (r_1,\dots, r_{d-1})}} Y$ satisfy $\big\langle Y,\tilde{Y}\big\rangle = \big\lVert \tilde{Y} \big\rVert^2$, then the angle condition simplifies to
\begin{equation}
\label{eq:ApproximateProjectionAngleCondition2}
\big\lVert \tilde{Y} \big\rVert \ge \omega \Big\lVert \mathcal{P}_{\mathbb{R}^{n_1 \times \cdots  \times n_d}_{\leq (r_1,\dots, r_{d-1})}}Y \Big\rVert.
\end{equation}
We use this simplified angle condition in \Cref{sec:rank_red_gen}.

\section{Methodology}
In this section, our main contributions to the RRAM -- the method to increase and decrease the rank -- are given. For both methods numerical experiments are given to illustrate the use and added value of the methods. 

\subsection{Rank increase} \label{sec:rank_incr_sec}

In \Cref{sec:approx_proj_gen}, we propose a simplified projection onto the tangent cone to determine the search direction. Afterwards, in \Cref{sec:rank_est_dim_d} we propose a method to determine an adequate value for the upper bound on the rank for the LRTCP. This involves a theoretical result for the LRTAP. We define an \emph{estimated rank} to extend this result to the LRTCP. Then in \Cref{sec:exp_rank_est_gen}, some experiments are given to illustrate the use of the rank estimation method. Lastly, in \Cref{sec:rank_incr} we give the full method that is used in the RRAM to increase the rank.
 
\subsubsection{Search direction in the tangent cone} \label{sec:approx_proj_gen}
Inspired by \cite{gao2022riemannian}, the proposed RRAM uses a search direction selected in the normal part of the tangent cone  with respect to the tangent space, which is the closed cone
\begin{equation*}
T_X^\perp\mathbb{R}_{\le (k_1, \dots, k_{d-1})}^{n_1 \times \cdots \times n_d}  = T_X\mathbb{R}_{\le (k_1, \dots ,k_{d-1})}^{n_1 \times \cdots \times n_d}  \cap \left(T_X\mathbb{R}_{(r_1, \dots, r_{d-1})}^{n_1 \times \cdots \times n_d} \right)^\perp,
\end{equation*}
to increase the rank. However, as can be seen from \eqref{eq:G_expanded_d}, the terms in the normal part of the tangent cone are highly non-linear and non-convex. In \cite{approx_proj_3D}, an approximate projection onto the tangent cone for third-order tensors was proposed. However, for higher-order tensors, such an approximate projection gets increasingly complex and computationally expensive. Thus, we propose to project iteratively onto the subcones
\begin{equation*}
\begin{split}
T_X^\perp \mathbb{R}_{\le (r_1, \dots, r_{i-1}, k_i, r_{i+1}, \dots, r_{d-1})}^{n_1 \times \cdots \times n_d} 
= \Big\lbrace& X_{1:(i-1)}' \cdot U_{i} \cdot V_{i+1} \cdot X_{(i+2):d}'' \Bigm\vert \mathrm{s.t.}~ U_i^{\mathrm{R},\top}X_i'^{\mathrm{R}}=0,~V_{i+1}^{\mathrm{L}}{X_{i+1}''}^{\mathrm{L},\top}=0 \Big\rbrace,
\end{split}
\end{equation*} 
for $i=1, \dots ,d-1$. This tangent cone is written more shortly as $T_X^\perp \mathbb{R}_{\le (r_{1:(i-1)}, k_i, r_{(i+1):d})}^{n_1 \times \cdots \times n_d}$. The projection is much easier as each element is only a second-order function in the parameters. The projection is given in the following proposition.

\begin{proposition} \label{prop:proj_normal_part} Let $X \in \mathbb{R}_{\le (r_1, \dots ,r_{d-1})}^{n_1 \times \cdots \times n_d}$ as in \eqref{eq:i-orth} and $Y \in \mathbb{R}^{n_1 \times \cdots \times n_d}$. A projection onto the tangent cone:
\begin{align*}
\mathcal{P}_{T_X^\perp \mathbb{R}_{\le (r_{1:(i-1)}, k_i, r_{(i+1):d} )}^{n_1 \times \cdots \times n_d}}&: \mathbb{R}^{n_1 \times \cdots \times n_d} \multimap T_X^\perp \mathbb{R}_{\le (r_{1:(i-1)}, k_i, r_{(i+1):d})}^{n_1 \times \cdots \times n_d}: Y \rightarrow \hat{Y},
\end{align*}
where $i \in \lbrace 1,\dots,d-1 \rbrace$, and 
$$\hat{Y} = X_{1:(i-1)}' \cdot U_i \cdot V_{i+1} \cdot X_{(i+2):d}'' \in \mathcal{P}_{T_X^\perp \mathbb{R}_{\le (r_{1:(i-1)}, k_i, r_{(i+1):d})}^{n_1 \times \cdots \times n_d}} Y,$$
is given by 
\begin{align} \label{eq:trunc_SVD_Pi}
\big[ U_i^\mathrm{R},S,V \big] &= \mathrm{SVD}_{s_i} \left( P_i \left( X,Y \right) \right), & V_{i+1}^\mathrm{L} &= S V^\top,
\end{align}
where $s_i:=k_i - r_i$, and
\begin{align} \label{eq:Pi}
&P_i(X,Y) :=  P_{X_i'{^\mathrm{R}}}^\perp \left[\left(X_{1:(i-1)}'\right)^{\mathrm{R},\top} Y^{<i-1,i+1>} \left( X_{(i+2):d}'' \right)^{\mathrm{L},\top} \right]^{r_{i-1}n_i \times n_{i+1} r_{i} } P_{{X_{i+1}''}^{\mathrm{L},\top}}^\perp.
\end{align}
\end{proposition}
\begin{proof}
Because of \eqref{eq:proj_max_reform_gen2}, $\hat{Y} \in \mathcal{P}_{T_X^\perp \mathbb{R}_{\le (r_{1:(i-1)}, k_i, r_{(i+1):d})}^{n_1 \times \cdots \times n_d}}$ maximizes the following inner product
\begin{align*}
\Big\lbrace U_i, V_{i+1} \Big\rbrace &= \argmin_{\substack{U,V \\ }}
\Big\langle Y, X_{1:(i-1)}' \cdot U \cdot V \cdot X_{(i+2):d}'' \Big\rangle,
\end{align*}
s.t. $\big\langle Y, X_{1:(i-1)}' \cdot U \cdot V \cdot X_{(i+2):d}'' \big\rangle = \big\lVert X_{1:(i-1)}' \cdot U \cdot V \cdot X_{(i+2):d}'' \big\rVert^2$ for all $X \in \mathbb{R}_{\le (r_1, \dots, r_{d-1})}^{n_1 \times \cdots \times n_d}$. 
Furthermore, any tensor $Y$ can be decomposed as
\begin{equation} \label{eq:Y_4terms_dim_d}
\begin{split}
Y^{<i-1,i+1>} =~ &P_{\left(X_{1:(i-1)}'\right)^{\mathrm{R}}} Y^{<i-1,i+1>} P_{\left(X_{(i+2):d}''\right)^{\mathrm{L},\top}} + P_{\left(X_{1:(i-1)}'\right)^{\mathrm{R}}}^\perp Y^{<i-1,i+1>} P_{\left(X_{(i+2):d}''\right)^{\mathrm{L},\top}} \\
&+ Y^{<i-1,i+1>} P_{\left(X_{(i+2):d}''\right)^{\mathrm{L},\top}}^\perp.
\end{split}
\end{equation}
The first term can further be decomposed as 
\begin{equation*}
\begin{split}
&P_{\left(X_{1:(i-1)}'\right)^{\mathrm{R}}} Y^{<i-1,i+1>} P_{\left(X_{(i+2):d}''\right)^{\mathrm{L},\top}} = \left(X_{1:(i-1)}'\right)^\mathrm{R} \Bigg[\\
  P_{X_i'{^\mathrm{R}}}^\perp &\left[ \left(X_{1:(i-1)}'\right)^{\mathrm{R},\top} Y^{<i-1,i+1>}  \left(X_{(i+2):d}''\right)^{\mathrm{L},\top} \right]^{r_{i-1}n_i \times n_{i+1} r_{i+1} } P_{{X_{i+1}''}^{\mathrm{L},\top}}^\perp +\\
 P_{X_i'{^\mathrm{R}}} &\left[ \left(X_{1:(i-1)}'\right)^{\mathrm{R},\top} Y^{<i-1,i+1>} \left(X_{(i+2):d}''\right)^{\mathrm{L},\top} \right]^{r_{i-1}n_i \times n_{i+1} r_{i+1} } P_{{X_{i+1}''}^{\mathrm{L},\top}}^\perp+\\
 &\left[ \left(X_{1:(i-1)}'\right)^{\mathrm{R},\top} Y^{<i-1,i+1>} \left(X_{(i+2):d}''\right)^{\mathrm{L},\top} \right]^{r_{i-1}n_i \times n_{i+1} r_{i+1} }  P_{{X_{i+1}''}^{\mathrm{L},\top}} \Bigg]^{r_{i-1} \times n_i n_{i+1} \times r_{i+i}} \left(X_{(i+2):d}''\right)^\mathrm{L},
\end{split}
\end{equation*}
which is a sum of three mutually orthogonal terms. If we insert this in \eqref{eq:Y_4terms_dim_d}, $Y$ can be written as the sum of five mutually orthogonal terms. Thus, the inner product with $\hat{Y} = X_{1:(i-1)}' \cdot U_i \cdot V_{i+1} \cdot X_{(i+2):d}''$ is
\begin{equation*}
\begin{split}
\big\langle \hat{Y},Y \big\rangle &= \Big\langle \hat{Y}^{<i-1,i+1>} , \left(X_{1:(i-1)}'\right)^\mathrm{R} P_i(X,Y) \left(X_{(i+2):d}''\right)^{\mathrm{L},\top} \Big\rangle  = \Big\langle \left[ U_i \cdot V_{i+1} \right]^{r_{i-1}n_i \times n_{i+1} r_{i} } ,  P_i(X,Y) \Big\rangle,
\end{split}
\end{equation*}
where we used the fact that the multiplication with the orthogonal matrices $\big(X_{1:(i-1)}'\big)^\mathrm{R}$ and $\big(X_{(i+2):d}''\big)^{\mathrm{L},\top}$ on the left and right of both terms does not change the inner product.
Because $\left[ U_i \cdot V_{i+1} \right]^{r_{i-1}n_i \times n_{i+1} r_{i} }$ is just a matrix of rank $s_i$, the optimal $U_i$ and $V_{i+1}$ can for example be obtained by the truncated SVD of $P_i(X,Y)$ as in \eqref{eq:trunc_SVD_Pi}.
\end{proof}

\subsubsection{Rank estimation} \label{sec:rank_est_dim_d}
In this subsection, a method to estimate how much the rank should be increased is proposed. This method is a generalization of the rank estimation method proposed for third-order tensors in \cite{rank_est_3D}. First, a theoretical result for the LRTAP is given and afterwards, we discuss how we use this result in the LRTCP.

\begin{theorem} \label{thm:local_min_best_approx_gen}
Let $X \in \mathbb{R}_{(r_1, \dots, r_{d-1})}^{n_1 \times \cdots \times n_d}$ and $A \in \mathbb{R}_{(r_1', \dots, r_{d-1}')}^{n_1 \times \dots \times n_d}$ with orthogonalizations as in \eqref{eq:i-orth}. If $\mathcal{P}_{T_{X}\mathbb{R}_{(r_1, \dots, r_{d-1})}^{n_1 \times \cdots \times n_d }} \nabla f \left(X \right)=0$, then 
\begin{align} \label{eq:Xi_local_min_dim_d}
\dot{X}_i &= B_i^\top \cdot \dot{A}_i \cdot C_i, & i &= 1, \dots, d,
\end{align}
where $B_1=I_{n_1}$, $C_d=I_{n_d}$, and
\begin{align*}
B_i^\top &= \left( X_{1:(i-1)}' \right)^{\mathrm{R},\top} \left( A_{1:(i-1)}'  \right)^{\mathrm{R}}, & C_i &= \left( A_{(i+1):d}''  \right)^{\mathrm{L}} \left( X_{(i+1):d}''  \right)^{\mathrm{L},\top}.
\end{align*}
And using \eqref{eq:def_Qi_Ri_orth_X_dim_d} and \eqref{eq:A_Qi}:
\begin{align} \label{eq:XvsA_local_min_dim_d}
X_i' &= B_{i}^\top \cdot A_i' \cdot Q_i' C_i Q_i^{-1}, & i &= 1, \dots, d-1, \nonumber \\
X_i'' &= Q_{i-1}^{-1} B_{i}^\top Q_{i-1}' \cdot A_i'' \cdot C_i, & i &= 2, \dots, d.
\end{align}
Thus, if we define the matrices
\begin{align*}
D_i &:= C_i Q_i^{-1} B_{i+1}^\top Q_{i}', & E_i &:= Q_i' C_i Q_i^{-1} B_{i+1}^\top, & i&=1,\dots,d-1,
\end{align*}
$X$ can be written as a function of the factors of $A$ as follows
\begin{align*}
X &= {A}'_1 E_1 \cdots A_{i-1}' \cdot E_{i-1} \cdot \dot{A}_{i} \cdot D_{i} \cdot A_{i+1}'' \cdots D_{d-1} A_d'', & i & \in \lbrace 1, \dots,d \rbrace.
\end{align*}
Furthermore,
$B_2 = Q_1' C_1 Q_1^{-1} \in \Stiefel(r_1,r_1')$, $C_{d-1} = \left( Q_{d-1}^{-1} B_{d}^\top Q_{d-1}' \right)^\top \in \Stiefel(r_{d-1},r_{d-1}')$, and $D_i$ and $E_i$ have $r_{i}$ times eigenvalue 1 and the other $r_{i}' - r_{i}$ eigenvalues are zero for all i. More specifically, 
\begin{align} \label{eq:Di_Ci_I}
B_{i+1}^\top Q_{i}' C_i Q_i^{-1} = Q_i^{-1} B_{i+1}^\top Q_i' C_i  = I_{r_i},
\end{align}
and thus $B_{i+1}^\top Q_{i}'$ and $C_i Q_i^{-1}$ are respectively right and left eigenvectors of $D_i$ corresponding to eigenvalue 1 and the same holds for $Q_i^{-1} B_{i+1}^\top$, $Q_i' C_i$, and $E_i$. 
\end{theorem}

\begin{proof}
From \eqref{eq:param_tan_space_dim_d}, and because $\mathcal{P}_{T_{X}\mathbb{R}_{(r_1, \dots, r_{d-1})}^{n_1 \times \cdots \times n_d }} \nabla f \left(X \right)=0$:
\begin{align}  \label{eq:cond_rank_est_dim_d_grad_Wi}
P_{{X_i'}^\mathrm{R}}^\perp \left(\left( X_{1:(i-1)}' \right)^{\mathrm{R},\top} \cdot A^{<i >} \cdot \left( X_{(i+1):d}''  \right)^{\mathrm{L},\top} \right)^\mathrm{R} &= 0, & i&=1,\dots, d-1 \nonumber \\
\dot{X}_d - \left( X_{1:(d-1)}'  \right)^{\mathrm{R},\top} A^\mathrm{R} &= 0. 
\end{align}
Using the last equation, we have:
\begin{equation*}
\dot{X}_d = \left( X_{1:(d-1)}'  \right)^{\mathrm{R},\top} \left( A_{1:(d-1)}' \right)^{\mathrm{R}} \dot{A}_d.
\end{equation*}
The factor $B_{d}^\top := \left( X_{1:(d-1)}'  \right)^{\mathrm{R},\top} \left( A_{1:(d-1)}' \right)^{\mathrm{R}}$ is a full rank matrix of dimension $r_{d-1} \times r_{d-1}'$ and using \eqref{eq:def_Qi_Ri_orth_X_dim_d} and \eqref{eq:A_Qi}:
\begin{equation} \label{eq:Xd_Ad_orth}
X_d'' = Q_{d-1}^{-1} B_{d}^\top Q_{d-1}' A_d'',
\end{equation}
satisfying \eqref{eq:XvsA_local_min_dim_d} and because $X_d''$ are $A_d''$ are orthogonal matrices, it holds that $\left( Q_{d-1}^{-1} B_{d}^\top Q_{d-1}' \right)^\top \in \Stiefel (r_{d-1},r_{d-1}')$. 

Now, we make use of the fact that for any of the orthogonality conditions in \eqref{eq:orth_cond_tangent_space_dim_d}, the parameters $W_j$ in \eqref{eq:param_tan_space_dim_d_gen} have to be zero. Suppose that $j = d-1$, then
\begin{align*}
\dot{X}_{d-1} - \left( X_{1:(d-2)}' \right)^{\mathrm{R},\top} \left( A_{1:(d-2)}' \right)^{\mathrm{R}} \cdot \dot{A}_{d-1} \cdot A_d'' {X_d''}^{\top} &= 0, 
\end{align*}
where again the factor $B_{d-1}^\top :=\left( X_{1:(d-2)}' \right)^{\mathrm{R},\top} \left( A_{1:(d-2)}' \right)^{\mathrm{R}}$ is a full rank matrix of size $r_{d-2} \times r_{d-2}'$. And using \eqref{eq:Xd_Ad_orth}, this equation can be simplified to
\begin{align*}
\dot{X}_{d-1} = B_{d-1}^\top \cdot \dot{A}_{d-1} \cdot \left( Q_{d-1}^{-1} B_{d}^\top Q_{d-1}' \right)^\top.
\end{align*}
Thus, 
$C_{d-1}=\left( Q_{d-1}^{-1} B_{d}^\top Q_{d-1}' \right)^\top$ and since $\left( Q_{d-1}^{-1} B_{d}^\top Q_{d-1}' \right)^\top \in \Stiefel (r_{d-1},r_{d-1}')$, $C_{d-1} Q_{d-1}^{-1} B_{d}^\top Q_{d-1}'$ has indeed $r_{d-1}$ times eigenvalue 1 and $r_{d-1}'-r_{d-1}$ times eigenvalue zero. Also for the other values of $j$, it directly follows that \eqref{eq:Xi_local_min_dim_d} holds by setting $W_j$ to zero in \eqref{eq:param_tan_space_dim_d_gen}. 

To prove that $D_i$ and $E_i$ have $r_{i}$ times eigenvalue 1 for all $i$, we use the fact that
\begin{align} \label{eq:Bi_ifv_Bi-1}
B_i^\top &= {X_{i-1}'}^{\mathrm{R},\top} \left( I_{n_{i-1}} \otimes \left(X_{1:(i-2)}'  \right)^{\mathrm{R},\top} \right) \left(  I_{n_{i-1}} \otimes \left(A_{1:(i-2)}'  \right)^{\mathrm{R}} \right) {A_{i-1}'}^{\mathrm{R}} \nonumber \\
&= {X_{i-1}'}^{\mathrm{R},\top} \left( I_{n_{i-1}} \otimes  B_{i-1}^\top \right) {A_{i-1}'}^{\mathrm{R}} = {X_{i-1}'}^{\mathrm{R},\top} \left( B_{i-1}^\top \cdot {A_{i-1}'} \right)^{\mathrm{R}}.
\end{align}
And thus,
\begin{align*}
D_{i-1} &=
C_{i-1} Q_{i-1}^{-1} {X_{i-1}'}^{\mathrm{R},\top} \left( B_{i-1}^\top \cdot {A_{i-1}'} \right)^{\mathrm{R}} Q_{i-1}'.
\end{align*}
If we multiply this matrix on the right with $C_{i-1} Q_{i-1}^{-1}$, we obtain
\begin{align*}
D_{i-1} C_{i-1} Q_{i-1}^{-1} =~&
C_{i-1} Q_{i-1}^{-1} {X_{i-1}'}^{\mathrm{R},\top} \left( B_{i-1}^\top \cdot {A_{i-1}'} \right)^{\mathrm{R}} Q_{i-1}' C_{i-1} Q_{i-1}^{-1} \\
=~& C_{i-1} Q_{i-1}^{-1} {X_{i-1}'}^{\mathrm{R},\top} {X_{i-1}'}^{\mathrm{R}} = C_{i-1} Q_{i-1}^{-1},
\end{align*}
and thus $C_{i-1} Q_{i-1}^{-1}$ are $r_{i-1}$ eigenvectors of $D_{i-1}$ with eigenvalue $1$. Since the matrix $D_{i-1}$ has rank $r_{i-1}$, the other eigenvalues have to be zero. Lastly, the eigenvalues of 
$$Q_{i-1}' D_{i-1} \left(Q'_{i-1}\right)^{-1} = Q_{i-1}' C_{i-1} Q_{i-1}^{-1} B_{i}^\top =E_{i-1}$$
are the same as those of $D_{i-1}$.
\end{proof}

The following proposition gives a formula for the rank difference between $X^{<i>}$ and $A^{<i>}$ in the case where the Riemannian gradient of the LRTAP at $X$ is zero.  This rank difference thus enables to determine the TT-rank of $A$.

\begin{proposition} \label{prop:rank_est_dim_d}
If the same conditions as in \Cref{thm:local_min_best_approx_gen} hold, then
\begin{align*}
&r_i'-r_i = \rank \left(\left[\left( X_{1:(i-1)}' \right)^{\mathrm{R},\top} \cdot \nabla f(X)^{<i-1,i+1 >} \cdot \left( X_{(i+2):d}''  \right)^{\mathrm{L},\top} \right]^{r_{i-1}n_i \times n_{i+1} r_{i+1}} \right),
\end{align*}
for $i=2,\dots,d-2$, and where $r_0:= r_d:=1$.
\end{proposition}
\begin{proof}
Using \Cref{thm:local_min_best_approx_gen}, it holds that
\begin{equation*}
\begin{split}
&\mathrm{rank} \left(\left[\left( X_{1:(i-1)}' \right)^{\mathrm{R},\top} \cdot \nabla f(X)^{<i-1,i+1 >} \cdot \left( X_{(i+2):d}''  \right)^{\mathrm{L},\top} \right]^{r_{i-1}n_i \times n_{i+1} r_{i+1}} \right) \\
&= \mathrm{rank} \left( \Big[ X_i' \cdot \dot{X}_{i+1} - B_i^\top \cdot A_i' \cdot \dot{A}_{i+1} \cdot C_{i+1} \Big]^{r_{i-1}n_i \times n_{i+1} r_{i+1}} \right) \\
&= \mathrm{rank} \left( \Big[ B_{i}^\top \cdot A_i' \cdot E_i \cdot \dot{A}_{i+1} \cdot C_{i+1} - B_i^\top \cdot A_i' \cdot \dot{A}_{i+1} \cdot C_{i+1} \Big]^{r_{i-1}n_i \times n_{i+1} r_{i+1}} \right) \\
&= \mathrm{rank} \left( \left[ B_{i}^\top \cdot A_i' \cdot \left( E_i - I_{r_i'} \right) \cdot \dot{A}_{i+1} \cdot C_{i+1} \right]^{r_{i-1}n_{i} \times n_{i+1} r_{i+1} } \right)\\
&= \mathrm{rank} \left( E_i - I_{r_i'}  \right) = r_i'-r_i.
\end{split}
\end{equation*}
Remark that the last equality holds because $E_i E_i = E_i$ and thus the space spanned by $E_i$ lies in the null space of $\left( E_i - I_{r_i'}  \right)$, which reduces the full rank $r_i'$ to $r_i'-r_i$.
\end{proof}

Remark that when $Y$ is the gradient $\nabla f(X)$, and the orthogonality conditions in \eqref{eq:param_tan_space_dim_d_gen} hold for $j=i+1$, then 
\begin{equation*}
\begin{split}
&\left[\left(X_{1:(i-1)}'\right)^{\mathrm{R},\top} \nabla f(X)^{<i-1,i+1>} \left(X_{(i+2):d}''\right)^{\mathrm{L},\top} \right]^{r_{i-1}n_i \times n_{i+1} r_{i} } =~ \\
&P_{X_i'{^\mathrm{R}}}^\perp \left[\left(X_{1:(i-1)}'\right)^{\mathrm{R},\top} \nabla f(X)^{<i-1,i+1>} \left(X_{(i+2):d}''\right)^{\mathrm{L},\top} \right]^{r_{i-1}n_i \times n_{i+1} r_{i} } P_{{X_{i+1}''}^{\mathrm{L},\top}}^\perp + \\
& P_{X_i'{^\mathrm{R}}} \left[\left(X_{1:(i-1)}'\right)^{\mathrm{R},\top} \nabla f(X)^{<i-1,i+1>} \left(X_{(i+2):d}''\right)^{\mathrm{L},\top} \right]^{r_{i-1}n_i \times n_{i+1} r_{i} } +  \\
& P_{X_i'{^\mathrm{R}}}^\perp \left[\left(X_{1:(i-1)}'\right)^{\mathrm{R},\top} \nabla f(X)^{<i-1,i+1>} \left(X_{(i+2):d}''\right)^{\mathrm{L},\top} \right]^{r_{i-1}n_i \times n_{i+1} r_{i} } P_{{X_{i+1}''}^{\mathrm{L},\top}} \\ &= P_i \left(X,\nabla f(X) \right)  + {X_i'}^\mathrm{R} W_{i+1}^\mathrm{L} + W_i^\mathrm{R} {X_{i+1}''}^\mathrm{L}.
\end{split}
\end{equation*}
Thus, when the Riemannian gradient is zero, it holds that
\begin{equation*}
\begin{split}
&\left[\left(X_{1:(i-1)}'\right)^{\mathrm{R},\top} \nabla f(X)^{<i-1,i+1>} \left(X_{(i+2):d}''\right)^{\mathrm{L},\top} \right]^{r_{i-1}n_i \times n_{i+1} r_{i} } = P_i \left(X,\nabla f(X) \right).
\end{split}
\end{equation*} 
Consequently, if the gradient is not exactly zero, it might be a good idea in practice to use $P_i(X,\nabla f(X))$ to estimate the rank $A$ instead of the expression proposed in \Cref{prop:rank_est_dim_d}.

To extend the result of \Cref{prop:rank_est_dim_d} to the LRTCP, we use the following definition of \emph{estimated rank} of a matrix inspired by \cite{gao2022riemannian}.

\begin{definition}[Estimated rank]
Given $B \in \mathbb{R}^{m\times n}$ and $ s< \mathrm{rank} \left( B \right)$, the estimated rank is defined as
\begin{equation} \label{eq:num_rank_increase}
\begin{split}
&\tilde{r}_s \left( B \right) := \begin{cases} 
0 & \text{if } B = 0 \\
\argmax_{\substack{j\le s}} \frac{\sigma_j\left( B\right) - \sigma_{j+1}\left( B \right)}{\sigma_j \left( B\right)}& \text{otherwise}
\end{cases},
\end{split}
\end{equation}
where $\sigma_{j}\left( B \right)$, $j= 1, \dots, \rank(B)$, denote the singular values of $B$ in decreasing order, i.e., $\sigma_i(B) \geq \sigma_j(B)$ for $i \leq j$. 
\end{definition}

The upper bound $s$ in the definition above prevents the estimated rank from being too high and should be chosen by the user. 

To obtain a stationary point on the manifold we optimize:
\begin{equation} \label{eq:fixed_rank_LRTCP}
\min_{X \in \mathbb{R}_{ (r_1,\dots, r_{d-1})}^{n_1 \times \cdots \times n_d}} {\frac{1}{2} \lVert X_{\Omega} - A_{\Omega} \rVert^2},
\end{equation}
using a Riemannian conjugate gradient algorithm developed for tensor completion in the TT format \cite{Steinl_high_dim_TT_compl_2016}.

An overview of the method to estimate the rank of a tensor $A$ based on the sparse tensor $A_\Omega $, is given in \Cref{alg:rank_est_dim_d}. The next subsection shows some numerical experiments using this algorithm.

\begin{algorithm}[h]
\setstretch{1.3} 
\caption{Rank estimation method for the LRTCP \eqref{eq:min_completion_TT}.} \label{alg:rank_est_dim_d}
\begin{algorithmic}[1]
\Require $\Omega \subseteq \lbrace 1, \dots, n_1 \rbrace \times \cdots \times \lbrace 1, \dots, n_d \rbrace,A_\Omega, r_i \in \mathbb{N}_0~ \forall i\in \lbrace 1,\dots,d-1 \rbrace$
\State 
$X^* \gets \argmin_{X \in \mathbb{R}_{ (r_1,\dots, r_{d-1})}^{n_1 \times \cdots \times n_d}} {\frac{1}{2} \lVert X_{\Omega} - A_{\Omega} \rVert^2}$;
\For{$i=1,\dots,d-1$}
\State $k_i \gets r_i + \tilde{r}_s \big( P_i \left( X^*, \nabla f_{\Omega}(X^*) \right) \big)$ \eqref{eq:Pi},\eqref{eq:num_rank_increase};
\EndFor
\Ensure $(k_1,\dots,k_{d-1})$
\end{algorithmic}
\end{algorithm}

\subsubsection{Experiments} \label{sec:exp_rank_est_gen}

In this section, we show the results of some numerical experiments to illustrate the use of \Cref{alg:rank_est_dim_d}. 
For these experiments, $A \in \mathbb{R}_{(r_1', \dots, r_{d-1}')}^{n_1 \times \cdots  \times n_d}$ is generated as follows:
\begin{equation} \label{eq:A_randn_d}
\begin{split}
A = ~&\texttt{randn} \left( n_1 \times r_1' \right) \cdot \texttt{randn} \left( r_1' \times n_2 \times r_2' \right) \cdots \texttt{randn} \left( r_{d-2}' \times n_{d-1} \times r_{d-1}' \right) \cdot \texttt{randn} \left( r_{d-1}' \times n_{d}\right).
\end{split}
\end{equation}
It can be shown that $A$, generated in this way, has approximately standard deviation $\sqrt{r_1' \cdots r_{d-1}'}$. To obtain $A_\Omega$, we randomly sample $A$ using \texttt{randperm} in {\Matlab}. We solve \eqref{eq:fixed_rank_LRTCP} using a Riemannian CG algorithm \cite{Steinl_high_dim_TT_compl_2016,manopt}. The starting point $X_0$ was generated randomly in $\mathbb{R}_{(r_1, \dots ,r_{d-1})}^{n_1 \times \cdots  \times n_d}$ in the same way as $A$ in \eqref{eq:A_randn_d}.

The results of the first experiment are shown in \Cref{fig:rank_est_d4_ni15_rAi3_rxi1_87it_gradR10-11_0_3prodn}. 
For this experiment, we choose $d:=4$, $n_i:=15$, for $i=1,\dots, d$, $r_i':=3$ and $r_i:=1$, for $i=1,\dots, d-1$, and $\rho_{\Omega} = 0.3$. The value of the Riemannian gradient at $X^*$ that is obtained for this experiment is approximately $10^{-11}$ after 87 iterations. The leftmost column of subfigures shows the first seven singular values of the matrices $P_i \left(X^*,\nabla f_\Omega(X^*)\right)$, for $i=1,\dots,d-1$. The second column shows the relative gap between these singular values, which is used the compute the estimated rank. As can be seen, the combination of the estimated ranks of these matrices indeed equals the difference between the TT-rank of $A$ and $X^*$. On the other hand, in the third and fourth column the singular values and relative gap of the unfoldings of $A_{\Omega}$ are respectively shown. As can be seen, based on the estimated rank of these unfoldings, the rank of $A$ can not be estimated. Furthermore, computing the singular values of the unfoldings of $A_{\Omega}$ is computationally more expensive ($\mathcal{O}\left(\max_i(n_i)^d \right)$ instead of $\mathcal{O}\left(\max_i(n_i)^3 \right)$). As a reference, in the rightmost column the singular values of the unfoldings of $A$ are shown to illustrate the true TT-rank of $A$. 

\begin{figure}[H]
\centering
\includegraphics[width= 0.85\linewidth]{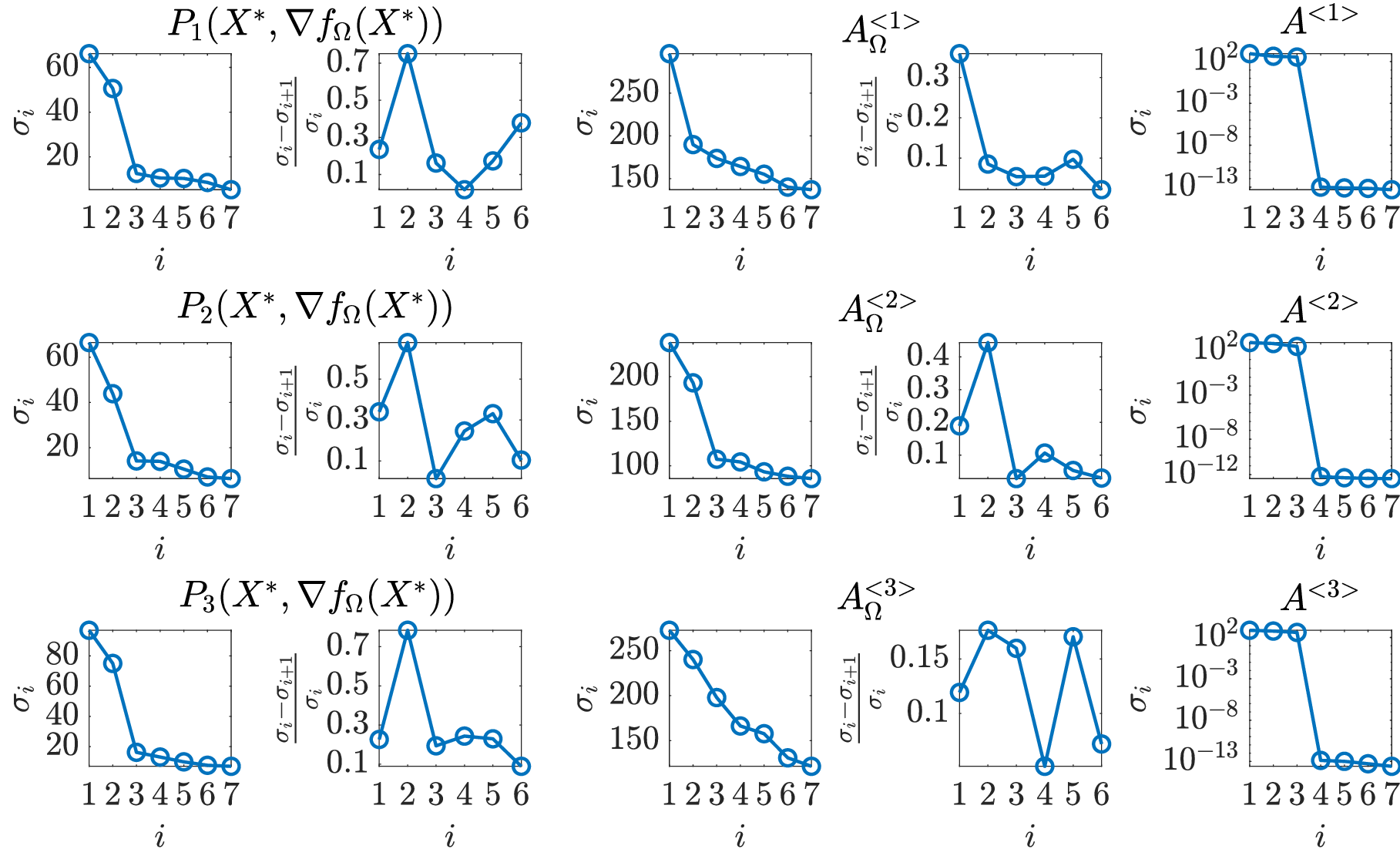}
\caption{The first 7 singular values of  
the matrices in \eqref{eq:Pi}, the relative gap between these singular values, the singular values of the unfoldings of $A_\Omega$ and $A$, where $X^*$ is obtained after 87 iterations of the CG algorithm, and with $d:=4$, $n_i:=15$, for $i=1,\dots,d$, $r'_i:=3$ and $r_i=1$, for $i=1,\dots,d-1$, and $\rho_{\Omega}:= 0.3$. The norm of the Riemannian gradient that is obtained at $X^*$ is approximately $10^{-11}$. }
\label{fig:rank_est_d4_ni15_rAi3_rxi1_87it_gradR10-11_0_3prodn}
\end{figure}

A second example is shown in \Cref{fig:rank_est_d4_ni15_rAi3_rxi1_15it_gradRn_31_0_3prodn}, where the number of iterations is lowered to $15$. The norm of the gradient that is obtained in this case is approximately $31$. The results are very similar and the estimated ranks of the matrices $P_i \left(X^*,\nabla f_\Omega(X^*)\right)$ still form the TT-rank of $A$ minus the TT-rank of $X^*$. A low norm of the gradient and consequently a high number of iterations is in practice thus not necessary in \Cref{alg:rank_est_dim_d} to estimate the TT-rank of $A$. Remark that the singular values of the unfoldings of $A_{\Omega}$ and $A$ have not changed compared to the previous experiment. 

\begin{figure}[H]
\centering
\includegraphics[width= 0.85\linewidth]{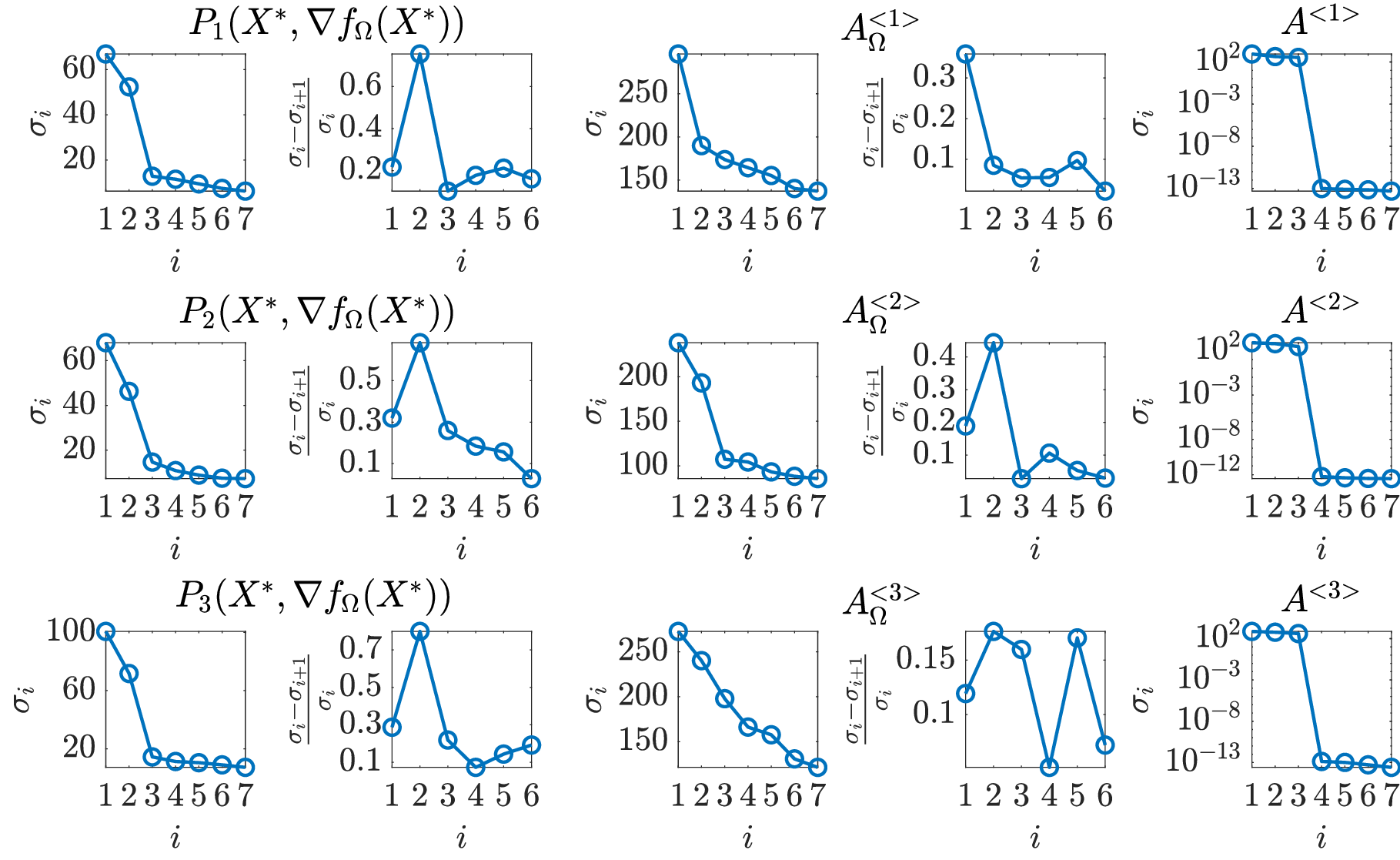}
\caption{The first 7 singular values of 
the matrices in \eqref{eq:Pi}, the relative gap between these singular values, the singular values of the unfoldings of $A_\Omega$ and $A$, where $X^*$ is obtained after 15 iterations of the CG algorithm, and with $d:=4$, $n_i:=15$, for $i=1,\dots,d$, $r'_i:=3$, $r_i:=1$, for $i=1,\dots,d-1$, and $\rho_{\Omega}:= 0.3$. The norm of the Riemannian gradient that is obtained at $X^*$ is approximately $31$. }
\label{fig:rank_est_d4_ni15_rAi3_rxi1_15it_gradRn_31_0_3prodn}
\end{figure}

In a third experiment the rank of $A$ is changed to $[2,5,3]$ and the results are shown in \Cref{fig:rank_est_d4_ni15_rAi253_rxi1_15it_gradRn13_0_3prodn}. 
The other parameters and maximal number of iterations were the same as in the previous experiment. The norm of the Riemannian gradient that is obtained is approximately 13.
Again only the matrices in \eqref{eq:Pi} can be used to determine the TT-rank of $A$. 

\begin{figure}[H]
\centering
\includegraphics[width= 0.85\linewidth]{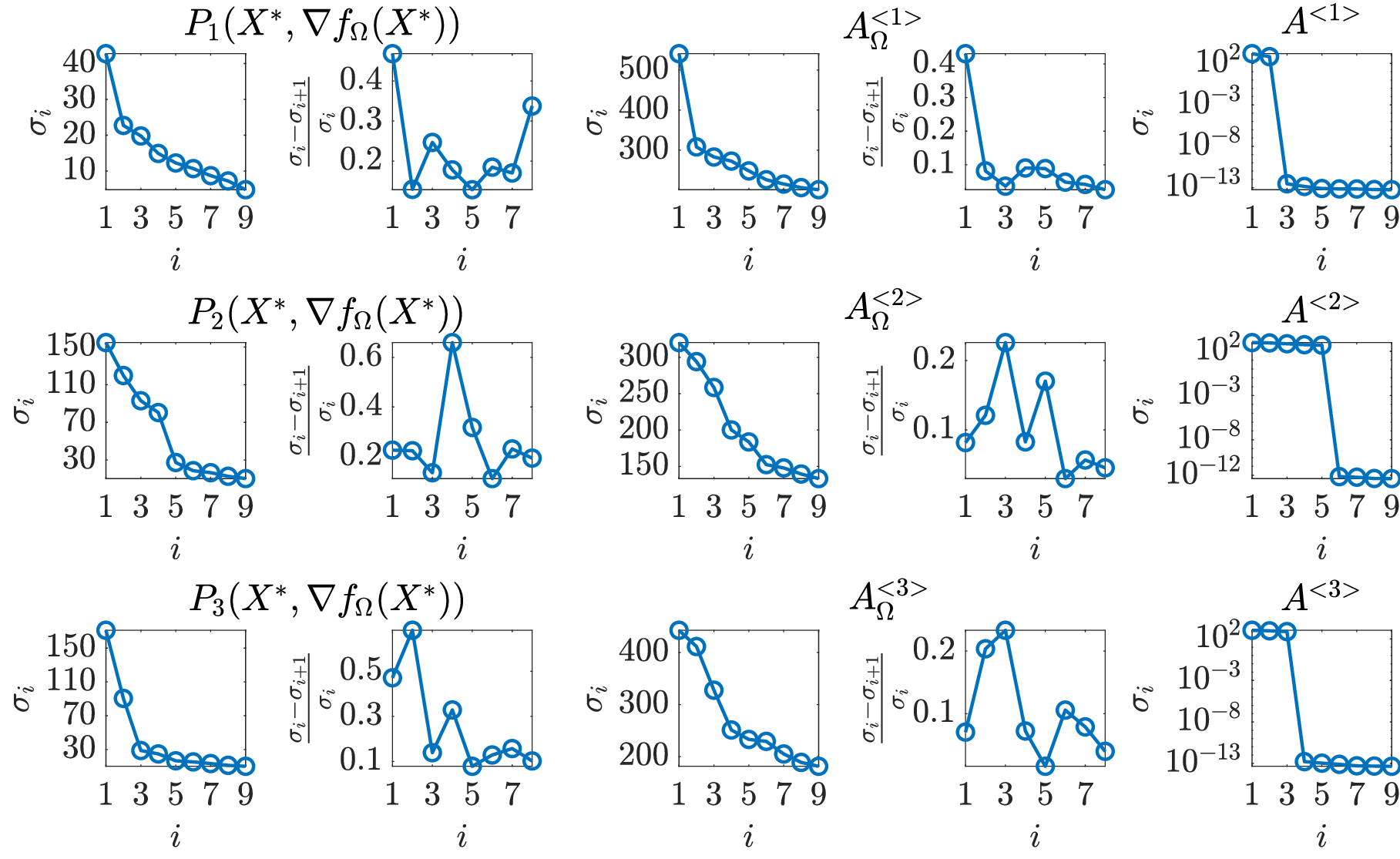}
\caption{The first 7 singular values of 
the matrices in \eqref{eq:Pi}, the relative gap between these singular values, the singular values of the unfoldings of $A_\Omega$, and $A$, where $X^*$ is obtained after 15 iterations of the CG algorithm, and with $d:=4$, $n_i:=15$, for $i=1,\dots,d$, $r':=[2,5,3]$, $r_i:=1$, for $i=1,\dots,d-1$, and $\rho_{\Omega}:= 0.3$. The norm of the Riemannian gradient that is obtained at $X^*$ is approximately $13$. }
\label{fig:rank_est_d4_ni15_rAi253_rxi1_15it_gradRn13_0_3prodn}
\end{figure}

In a fourth experiment, the dimension is increased to $d:=5$ and the value of $n_i$ is lowered to 10 for all $i=1,\dots,d$. The other parameters were the same as in the second experiment. The CG algorithm is run for 15 iterations and the norm of the Riemannian gradient is approximately $12$. The results are shown in \Cref{fig:rank_est_d5_ni10_rAi3_rxi1_15it_gradRn_12_0_3prodn}. Again, only \Cref{alg:rank_est_dim_d} is able to determine the exact TT-rank of $A$ based on the sampled elements.

\begin{figure}[H]
\centering
\includegraphics[width= 0.85\linewidth]{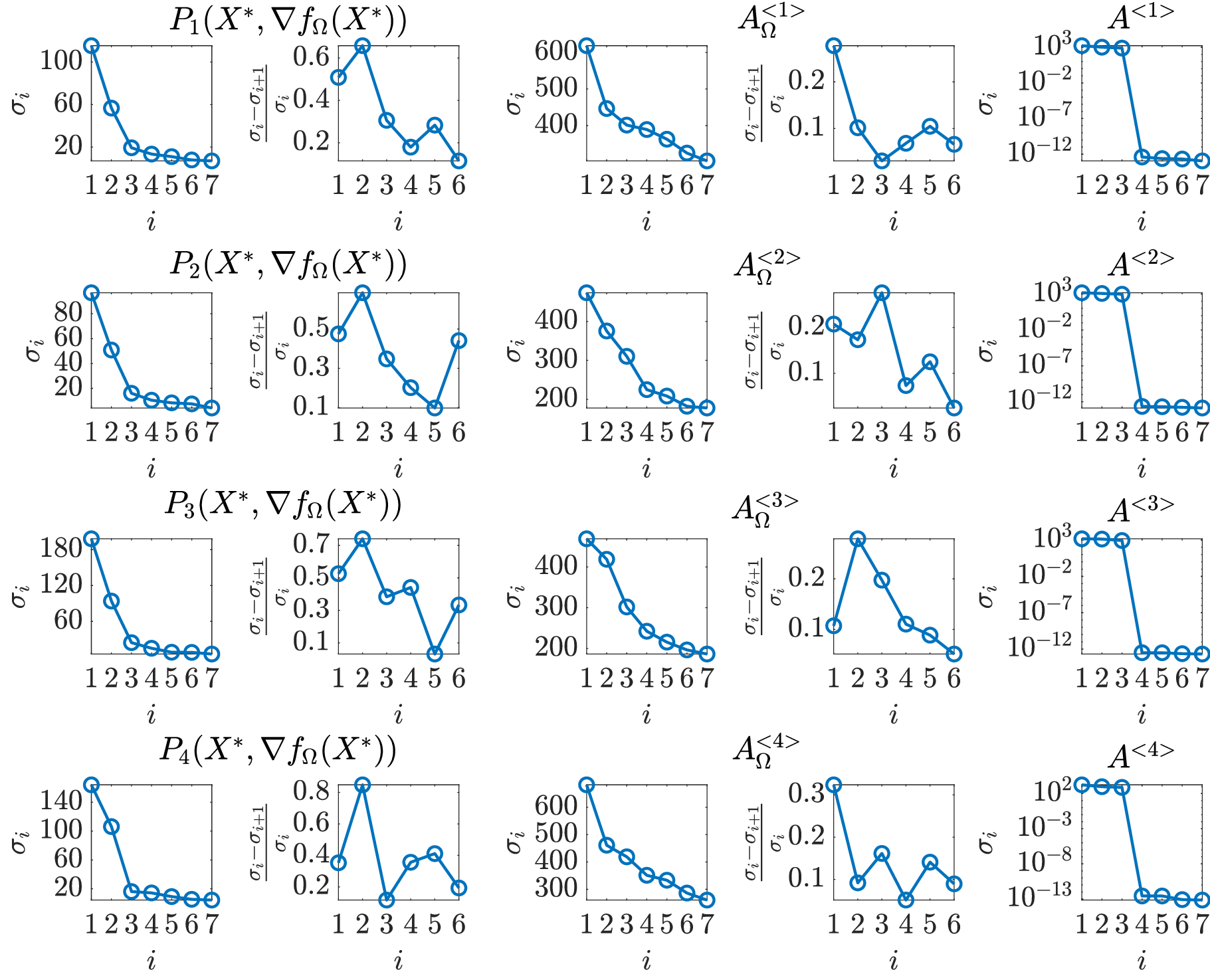}
\caption{The first 7 singular values of 
the matrices in \eqref{eq:Pi}, the relative gap between these singular values, the singular values of the unfoldings of $A_\Omega$, and $A$, where $X$ is obtained after 15 iterations of the CG algorithm, and with $d:=5$, $n_i:=10$, for $i=1,\dots,d$, $r'_i:=3$ and $r_i:=1$, for $i=1,\dots,d-1$, and $\rho_{\Omega}:= 0.3$. The norm of the Riemannian gradient that is obtained at $X^*$ is approximately $12$.}
\label{fig:rank_est_d5_ni10_rAi3_rxi1_15it_gradRn_12_0_3prodn}
\end{figure}

In a last experiment, noise is added to the data as follows:
\begin{equation} \label{eq:A_noise}
A_{\eta} := A + \eta~ \mlin{randn} \left(n_1 \times \cdots \times n_d \right),
\end{equation}
for $\eta := 10^{-1}$. The results are shown in \Cref{fig:rank_est_d5_ni10_rAi3_rxi1_15it_gradRn_77_0_3prodn_0_1noise}. As can be seen based on the rightmost column, $A$ is now indeed only approximately of low TT-rank but this does not affect the estimated TT-rank from \Cref{alg:rank_est_dim_d}.

\begin{figure}[H]
\centering
\includegraphics[width= 0.85\linewidth]{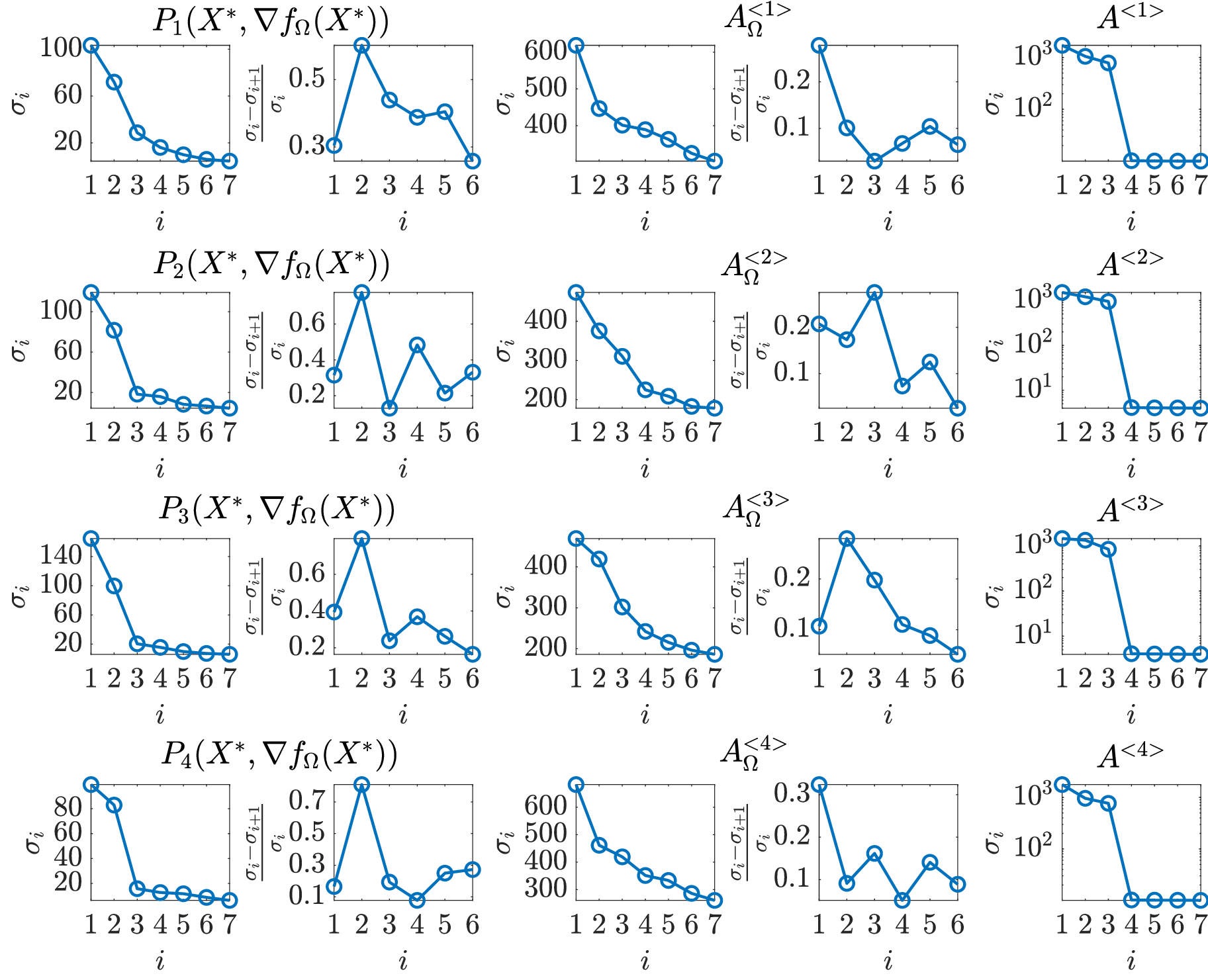}
\caption{The first 7 singular values of 
the matrices in \eqref{eq:Pi}, the relative gap between these singular values, the singular values of the unfoldings of $A_\Omega$, and $A$, where noise is added to the data as in \eqref{eq:A_noise} of size $10^{-1}$, and $X$ is obtained after 15 iterations of the CG algorithm, and with $d:=5$, $n_i:=10$, for $i=1,\dots,d$, $r'_i:=3$ and $r_i:=1$, for $i=1,\dots,d-1$, and $\rho_{\Omega}:= 0.3$. The norm of the Riemannian gradient that is obtained at $X^*$ is approximately $77$.}
\label{fig:rank_est_d5_ni10_rAi3_rxi1_15it_gradRn_77_0_3prodn_0_1noise}
\end{figure}

\subsubsection{Algorithm for rank increase} \label{sec:rank_incr}
When the search direction $\hat{Y}$ is determined, a line search along this direction is performed. For the LRTCP \eqref{eq:min_completion_TT}, we can perform an exact line search, i.e., we can compute the step size $t_k$ at iteration $k$ as $t_k = \argmin_{t>0} f_{\Omega} \big(X^{(k)} + t  \hat{Y} \big)$, which can be obtained as
\begin{equation} \label{eq:opt_step_size}
\begin{split}
\frac{\partial f_{\Omega} \big( X^{(k)} + t  \hat{Y} \big)}{\partial t} = \frac{\partial}{\partial t} \left( \frac{1}{2} \Big\lVert X^{(k)}_{\Omega} + t_k \hat{Y}_{\Omega} - A_{\Omega} \Big\rVert^2 \right) = 0 & \Leftrightarrow \Big\langle X^{(k)}_{\Omega} + t_k \hat{Y}_{\Omega} - A_{\Omega}, \hat{Y}_{\Omega} \Big\rangle = 0, \\
&\Leftrightarrow t_k = -\frac{\big\langle \nabla f_\Omega \big(X^{(k)} \big), \hat{Y}_{\Omega} \big\rangle}{\big\lVert \hat{Y}_{\Omega} \big\rVert^2},
\end{split}
\end{equation}
where $\hat{Y}_{\Omega}$ is defined by \eqref{eq:Z_Omega} and $X^{(k)}$ is the current best approximation at iteration $k$.

To increase the rank $r_i$, using the parameters $U_i$ and $V_{i+1}$ that we obtain from the projection in \Cref{prop:proj_normal_part}, we need the following retraction operator \cite{kutschan2018} which projects $X + t \hat{Y}$ to the manifold $ \mathbb{R}_{(r_1, \dots, r_{i-1}, k_i, r_{i+1}, \dots, r_{d-1})}^{n_1 \times \cdots \times n_d}$:
\begin{equation} \label{eq:retraction_high_dim}
\begin{split}
\mathcal{R}\left( X, t, \hat{Y} \right) =  X_{1:(i-1)}' \cdot \begin{bmatrix} X_i' & U_i \end{bmatrix} \cdot \begin{bmatrix} X_{i+1} \\ t V_{i+1} \end{bmatrix} \cdot X_{(i+2):d}'' ,
\end{split}
\end{equation}
where $\hat{Y} = X_{1:(i-1)}' \cdot U_i \cdot V_{i+1} \cdot X_{(i+2):d}'' \in {T_X^\perp \mathbb{R}_{\le (r_{1:(i-1)}, k_i, r_{(i+1):d})}^{n_1 \times \cdots \times n_d}}$.

We do this for $i=1,\dots, d-1$. For each $i$, the rank increase is determined by the estimated rank from \Cref{prop:rank_est_dim_d}. An overview of the algorithm is given in \Cref{alg:rank_incr_dim_d}. Remark that the rank is only increased if the direction produces a sufficient decrease in the cost and test function. This is controlled by a small parameter $\varepsilon$.

\begin{algorithm}[H]
\setstretch{1.3} 
\caption{Method to increase the rank of a TTD in a RRAM based on projections of the negative gradient onto subcones of the tangent cone.}
\label{alg:rank_incr_dim_d}
\begin{algorithmic}[1]
\Require
$\Omega, \Gamma, X \in \mathbb{R}_{\le (r_1,\dots, r_{d-1})}^{n_1 \times \cdots \times n_d}, s_{\max},r_{\max} \in \mathbb{N}_0^{d-1}, 0 \leq \varepsilon < 1$
\For{$i=1,\dots,d-1$}
\State $P \gets P_i\left(X,\nabla f_\Omega(X) \right)$ \eqref{eq:Pi};
\State $s \gets \min \left(r_{\max}(i)-r_i,s_{\max}(i) \right)$;
\State
$s_i \gets \tilde{r}_{s} \left( P \right)$  \eqref{eq:num_rank_increase};
\State
$[U_i^\mathrm{R},S,V] \gets \mathrm{SVD}_{s}\left( P \right)$; 
\State $V_{i+1}^\mathrm{L} \gets S V^\top$;
\State $\hat{Y} \gets X = X'_{1:(i-1)} \cdot X_i' \cdot \dot{X}_{i+1} \cdot X''_{(i+2):d}$; 
\State $\hat{Y}_i \gets U_i$; 
\State $\hat{Y}_{i+1} \gets V_{i+1}$;
\State $t \gets - \frac{\langle \hat{Y}_\Omega, \nabla f_\Omega \left( X\right) \rangle}{\lVert \hat{Y}_\Omega \rVert^2} $; 
\State $X_{\mathrm{new}} \gets \mathcal{R} ( X, t,\hat{Y} )$ \eqref{eq:retraction_high_dim};
\If{$\left( f_\Omega \left(X\right) -f_\Omega \left(X_{\mathrm{new}}\right)\right) > \varepsilon \And \left( f_\Gamma \left(X\right) -f_\Gamma \left(X_{\mathrm{new}}\right)\right) > \varepsilon$}
\State $X \gets X_{\mathrm{new}}$;
\EndIf
\EndFor
\Ensure
$X$
\end{algorithmic}
\end{algorithm}

\subsection{Rank reduction} \label{sec:rank_red_gen}

To reduce the rank, we propose to use the TT-rounding algorithm \cite[Algorithm 2]{Oseledets2011}, shown in \Cref{alg:TT_rounding}. The rank reduction is performed from left to right. Remark that we could also perform the reduction from right to left and compare the result but this would require double the computational effort and is thus not implemented. In the next theorem, we show that \Cref{alg:TT_rounding} can be considered as an approximate projection and satisfies a certain angle condition.

\begin{algorithm}[h]
\setstretch{1.3}
\caption{Rank reduction \cite[Algorithm 2]{Oseledets2011}}
\label{alg:TT_rounding}
\begin{algorithmic}[1] 
\Require
$A = \dot{A}_1 \cdot A_2'' \cdots A_d'' \in \mathbb{R}_{\le (r_1',\dots, r_{d-1}')}^{n_1 \times \cdots \times n_d},r_i \in \mathbb{N}_0~ \forall i \in \lbrace 1,\dots,d-1\rbrace$
\For{$i=1,\dots,d-1$}
\State $[{X_i'}^\mathrm{R},S_i,V_i] \gets \mathrm{SVD}_{r_i} \left(\dot{A}_i^\mathrm{R} \right)$;
\State $\dot{A}_{i+1} \gets S_i V_i^\top \cdot A_{i+1}''$;
\EndFor
\State $\dot{X}_d \gets \dot{A}_{d}$;
\Ensure $X = X_1' \cdots X_{d-1}' \cdot \dot{X}_d \in \mathbb{R}_{\le (r_1,\dots, r_{d-1})}^{n_1 \times \cdots \times n_d}$
\end{algorithmic}
\end{algorithm}

\begin{theorem}[Angle condition \Cref{alg:TT_rounding}] \label{thm:angle_cond_rank_red_gen}
Let $A \in \mathbb{R}_{\le (r_1',\dots, r_{d-1}')}^{n_1 \times \cdots  \times n_d}$. \Cref{alg:TT_rounding} can be considered as an approximate projection:
\begin{align*}
\tilde{\mathcal{P}}_{\mathbb{R}_{\le (r_1, \dots, r_{d-1})}^{n_1 \times \cdots \times n_d}}&: \mathbb{R}_{\le (r_1',\dots, r_{d-1}')}^{n_1 \times \cdots \times n_d} \multimap \mathbb{R}_{\le (r_1, \dots, r_{d-1})}^{n_1 \times \cdots \times n_d},
\end{align*}
that satisfies the necessary condition \eqref{eq:nec_cond_approx_proj} and thus satisfies the angle condition \eqref{eq:ApproximateProjectionAngleCondition2} with
\begin{equation} \label{eq:omega_rank_red_gen}
\omega = \sqrt{\frac{\prod_{i=2}^{d-1} r_i}{\prod_{i=2}^{d-1} r_i'}}.
\end{equation}
\end{theorem}
\begin{proof}
We first prove that every output $X$ of \Cref{alg:TT_rounding} satisfies the necessary condition \eqref{eq:nec_cond_approx_proj}. In the first iteration, it holds that $X_1' = \dot{A}_1 V_1 S_1^{-1}$ and in every subsequent iteration $i=2, \dots, d-1$ it holds that ${X_i'} = S_{i-1} V_{i-1}^\top \cdot {A''_{i}} \cdot V_i S_i^{-1}$,
and in the last iteration $\dot{X}_d = S_{d-1} V_{d-1} A_d''$. Thus, the $X$ that is given to the output can be written as
\begin{align*}
X = X_1' \cdots X_{d-1}' \cdot \dot{X}_d &= \dot{A}_1 V_1 S_1 S_{1}^{-1} V_1^\top \cdot A_2'' \cdot V_2 S_2 S_{2}^{-1} V_2^\top \cdot A_3'' \cdots V_{d-1} S_{d-1} S_{d-1}^{-1} V_{d-1}^\top A_d'' \\
&= \dot{A}_1 P_{V_1} \cdot A''_2 \cdot P_{V_2} \cdot A''_3 \cdots P_{V_{d-1}} A_d''.
\end{align*}
Furthermore, every tensor $A$ can be written as
\begin{align*}
A &= \dot{A}_1 \left( P_{V_1} + P_{V_1}^\perp\right) \cdot A''_2 \cdot \left( P_{V_2} + P_{V_2}^\perp\right) \cdot A''_3 \cdots \left( P_{V_{d-1}} + P_{V_{d-1}}^\perp \right) \cdot A_d'',
\end{align*}
which can be expanded into a sum of $2^{d-1}$ terms. For every term, except for the one that equals $X$, it holds that there exists an index $i \in \lbrace 1,\dots, d-1 \rbrace$ such that the first $i$ factors are
\begin{align*}
&\dot{A}_1 P_{V_1} \cdot A''_2 \cdots P_{V_{i-1}} \cdot A_i'' \cdot P_{V_i}^\perp = X_1' \cdots X_{i-1}' \cdot S_{i-1} V_{i-1}^\top \cdot A_i'' \cdot P_{V_i}^\perp = X_1' \cdots X_{i-1}' \cdot \dot{A}_i \cdot P_{V_i}^\perp.
\end{align*}
Thus, it holds that
\begin{align*}
&\left(X_1' \cdots X_i' \right)^{\mathrm{R},\top} \left( X_1' \cdots X_{i-1}' \cdot \dot{A}_i \cdot P_{V_i}^\perp \right)^\mathrm{R} \\
 & = {X_i'}^{\mathrm{R},\top} \left(I_{n_i} \otimes \left( X_1' \cdots X_{i-1}' \right)^{\mathrm{R},\top} \right) \left( I_{n_i} \otimes \left( X_1' \cdots X_{i-1}' \right)^\mathrm{R} \right) \left( \dot{A}_i \cdot P_{V_i}^\perp \right)^\mathrm{R}\\
& = {X_i'}^{\mathrm{R},\top} {\dot{A}_i}^\mathrm{R} P_{V_i}^\perp  = S_i V_i^\top P_{V_i}^\perp=0,
\end{align*}
and because of \eqref{eq:inner_prod_zero}, the inner product of $A$ with $X$ is $\langle X,A \rangle = \langle X,X \rangle$, and the angle condition can be written in the form \eqref{eq:ApproximateProjectionAngleCondition2}. 

Let $\hat{X} = \hat{X}_1' \cdots \hat{X}_{d-1}' \cdot \hat{\dot{X}}_d \in \mathcal{P}_{\mathbb{R}_{\le (r_1, \dots, r_{d-1})}^{n_1 \times \cdots \times n_d}}A$, then
$\mathcal{P}_{T_{\hat{X}}\mathbb{R}_{(r_1, \dots, r_{d-1})}^{n_1 \times \cdots \times n_d }} \nabla f \left(\hat{X} \right)=0$ and $\hat{X}$ can be written as in \Cref{thm:local_min_best_approx_gen}:
\begin{equation*}
\hat{X} = \dot{A}_1 \cdot D_1 \cdot A_2'' \cdots D_{d-1} \cdot A_d'',
\end{equation*}
where based on \eqref{eq:Xi_local_min_dim_d} and \eqref{eq:XvsA_local_min_dim_d}: $\hat{X}_1' = \dot{A}_1 C_1 Q_1^{-1}$, $\hat{X}_i' = B_i^\top Q_{i-1}' \cdot {A}_i'' \cdot C_i Q_i^{-1}$, for $i=2,\dots,d-1$, and $\hat{\dot{X}}_d = B_d^\top Q_{d-1}' {A}_d''$.
Furthermore, $\hat{X}$ must by definition also satisfy the necessary condition $\left\langle \hat{X},A \right\rangle = \left\lVert \hat{X} \right\rVert^2$. Additionally, $\left\lVert \hat{X} \right\rVert \leq \lVert A\rVert$, and more specifically
\begin{align} \label{eq:ineq_hatX}
\left\lVert \hat{X} \right\rVert = \left\lVert \dot{A}_1 \cdot D_1 \cdot A_2'' \cdots D_{d-1} \cdot A_d''\right\rVert \nonumber \leq~ & \left\lVert \dot{A}_1 \cdot D_1 \cdot A_2'' \cdots D_{d-2} \cdot A_{d-1}'' \cdot A_d''\right\rVert \nonumber\\
&\vdots \nonumber \\
\leq~ & \left\lVert \dot{A}_1 \cdot D_1 \cdot A_2'' \cdots A_d''\right\rVert.
\end{align}
These inequalities hold because 
\begin{align*}
\left\lVert \dot{A}_1 D_1 \cdot A_2'' \cdots D_{i} \cdot A_{i+1}'' \cdots A_d'' \right\rVert &= \left\lVert \dot{A}_1  D_1 \cdot A_2'' \cdot D_2 \cdots A_{i}'' \cdot D_{i} \right\rVert \\
&= \left\lVert \dot{A}_1 C_1 Q_1^{-1} B_{2}^\top Q_{1}' \cdot A_2'' \cdot C_2 Q_2^{-1} B_{3}^\top Q_{2}' \cdots A_{i}'' \cdot C_i Q_i^{-1} B_{i+1}^\top Q_{i}' \right\rVert \\
&= \left\lVert \hat{X}_1' \cdots \hat{X}_{i}' \cdot B_{i+1}^\top Q_i' \right\rVert \\
&= \left\lVert B_{i+1}^\top Q_i' \right\rVert,
\end{align*}
and from \eqref{eq:Bi_ifv_Bi-1} we know that $B_{i+1}^\top = {X_{i}'}^{\mathrm{R},\top} \left( B_{i}^\top \cdot {A_{i}'} \right)^{\mathrm{R}}$. Thus,
\begin{align*}
\left\lVert B_{i+1}^\top Q_i' \right\rVert = \left\lVert {X_{i}'}^{\mathrm{R},\top} \left( B_{i}^\top \cdot {A_{i}'} \right)^{\mathrm{R}} Q_i' \right\rVert = \left\lVert {X_{i}'}^{\mathrm{R},\top} \left( B_{i}^\top Q_{i-1}' \cdot {A_{i}''} \right)^{\mathrm{R}} \right\rVert &= \left\lVert P_{{X_{i}'}^{\mathrm{R}}} \left( B_{i}^\top Q_{i-1}' \cdot {A_{i}''} \right)^{\mathrm{R}} \right\rVert \\
&\leq \left\lVert B_{i}^\top Q_{i-1}' \cdot {A_{i}''}^{\mathrm{L}}  \right\rVert = \left\lVert B_{i}^\top Q_{i-1}' \right\rVert,
\end{align*}
and consequently \eqref{eq:ineq_hatX} holds. 

We can now prove the angle condition \eqref{eq:ApproximateProjectionAngleCondition2} with $\omega$ as in \eqref{eq:omega_rank_red_gen}. In the first iteration, ${X_1'}^\mathrm{R}$ is obtained from the truncated SVD of $\dot{A}_1$ and thus it holds that 
\begin{align} \label{eq:ineq_PV1_hatX}
\left\lVert \dot{A}_1 P_{V_1} \right\rVert = \left\lVert \dot{A}_1 P_{V_1} \cdot A_2'' \cdots A_d'' \right\rVert = \left\lVert P_{X_1'} A^\mathrm{L} \right\rVert \overset{\eqref{eq:ineqs_svd_trunc1}}{\geq} \left\lVert P_{\hat{X}_1'} A^\mathrm{L} \right\rVert &= \left\lVert \dot{A}_1 C_1 Q_1^{-1} \left(C_1 Q_1^{-1}\right)^\top Q_1'^\top A_1'^\top A_1' \cdot \dot{A}_2 \right\rVert \nonumber\\
& = \left\lVert \dot{A}_1 C_1 Q_1^{-1} B_2^\top Q_1' \cdot {A}_2'' \right\rVert \nonumber\\
&= \left\lVert \dot{A}_1 D_1 \right\rVert \overset{\eqref{eq:ineq_hatX}}{\geq} \left\lVert \hat{X} \right\rVert,
\end{align}
where we used the fact that $B_2 = Q_1'C_1 Q_1^{-1}$ as in \Cref{thm:local_min_best_approx_gen}.
In the next iteration, $V_2$ is obtained from the truncated SVD of $\left( S_1 V_1^\top  \cdot A_2'' \right)^\mathrm{R}$ and thus
\begin{align*}
\left\lVert \left( S_1 V_1^\top  \cdot A_2'' \right)^\mathrm{R} P_{V_2} \right\rVert^2 = \left\lVert X_1' S_1 V_1^\top  \cdot A_2'' \cdot P_{V_2} \right\rVert^2 &= \left\lVert \dot{A}_1 P_{V_1}  \cdot A_2'' \cdot P_{V_2} \right\rVert^2  \\
& \overset{\eqref{eq:ineqs_svd_trunc2}}{\geq} \frac{r_2}{r_2'} \left\lVert S_1 V_1^\top  \cdot A_2'' \right\rVert^2 = \frac{r_2}{r_2'} \left\lVert \dot{A}_1 P_{V_1} \right\rVert^2.
\end{align*}
In the third iteration, we can apply this principle again to obtain
\begin{align*}
 \lVert \dot{A}_1 P_{V_1} \cdot A_2'' \cdot P_{V_2} \cdot A_3'' \cdot P_{V_3} \rVert^2  \overset{\eqref{eq:ineqs_svd_trunc2}}{\geq} \frac{r_3}{r_3'} \lVert \dot{A}_1 P_{V_1} \cdot A_2'' \cdot P_{V_2} \rVert^2.
\end{align*}
We can do this recursively and combine the inequalities to obtain
\begin{align*}
\lVert \dot{A}_1 P_{V_1} \cdot A_2'' \cdot P_{V_2} \cdot A_3'' \cdots P_{V_{d-1}} A_d'' \rVert^2  \geq \frac{\prod_{i=2}^{d-1} r_i}{\prod_{i=2}^{d-1} r_i'} \lVert \dot{A}_1 P_{V_1} \cdot A_2'' \cdots A_d'' \rVert^2.
\end{align*}
And thus using \eqref{eq:ineq_PV1_hatX}, we obtain
\begin{align*}
\lVert \dot{A}_1 P_{V_1} \cdot A_2'' \cdot P_{V_2} \cdot A_3'' \cdots P_{V_{d-1}} A_d'' \rVert^2  = \lVert X \rVert^2   \geq \frac{\prod_{i=2}^{d-1} r_i}{\prod_{i=2}^{d-1} r_i'} \lVert \hat{X} \rVert^2,
\end{align*}
from which the angle condition follows.
\end{proof}

\subsubsection{Numerical rank}
To determine how much the rank should be decreased, the following \emph{numerical $\Delta$-rank} is defined, inspired by \cite{gao2022riemannian}.

\begin{definition}[Numerical matrix rank] \label{def:num_rank}
Given $\Delta \in [0, 1]$, the \emph{$\Delta$-rank} of $B \in \mathbb{R}^{m \times n}$ is defined as
\begin{equation*}
\rank_\Delta B :=
\left\{\begin{array}{ll}
0 & \text{if } B = 0\\
\min\left\{j \in \{1, \dots, \rank B\} \mid \frac{\sigma_j(B)-\sigma_{j+1}(B)}{\sigma_j(B)} \geq \Delta\right\} & \text{otherwise}
\end{array}\right.,
\end{equation*}
where as in \eqref{eq:num_rank_increase} $\sigma_{j}(B)$, for $j=1,\dots, \rank B$, denote the singular values of $B$ in decreasing order and $\sigma_{j+1}(B) = 0$ for $j = \rank B$. 
\end{definition}

For $\Delta:=0$, the $\Delta$-rank equals one and for $\Delta:=1$, it equals the standard rank. The following definition gives a generalization to tensors.

\begin{definition}[Numerical tensor rank] \label{def:num_rank_gen}
Given $\Delta \in [0, 1]$, the \emph{$\Delta$-rank} of $X \in \mathbb{R}^{n_1 \times \cdots \times n_d}$ is defined as
\begin{equation*}
\mathrm{rank}_{\Delta} \left(X\right) := \left( \mathrm{rank}_{\Delta} X^{<1>},~ \mathrm{rank}_{\Delta} X^{<2>},~ \dots,~ \mathrm{rank}_{\Delta} X^{<d-1>} \right),
\end{equation*}
where the $\Delta$-rank for matrices was defined in \Cref{def:num_rank}.
\end{definition}

Remark that the computation of the SVD of these unfoldings can be computationally expensive for high-dimensional tensors. However, when $X$ is already given as a TTD, the cost can be reduced significantly by using the terms $\dot{X}_i$. The $\Delta$-rank then becomes:
\begin{equation} \label{eq:num_rank_TTD}
\mathrm{rank}_{\Delta} \left( X \right) = \left( \mathrm{rank}_{\Delta} \dot{X}_1,~ \mathrm{rank}_{\Delta} \dot{X}_2^\mathrm{R},~ \dots,~ \mathrm{rank}_{\Delta} \dot{X}_{d-1}^\mathrm{R} \right).
\end{equation}
This is possible because
\begin{equation*}
X^{<i>} = \left( I_{n_i} \otimes \left( X_{1}' \cdots X_{i-1}' \right)^\mathrm{R} \right) \dot{X}_i^\mathrm{R} \left( X_{i+1}'' \cdots X_d'' \right)^\mathrm{L},
\end{equation*}
and thus the singular values of $X^{<i>}$ and $\dot{X}_i^\mathrm{R}$ are the same.
\subsubsection{Experiment}

An illustration of the angle condition in \Cref{thm:angle_cond_rank_red_gen} is shown in \Cref{fig:test_angle_cond_d4_rAi4_ri2_TT-rounding_ni10_A_TTeMPS_randn}. For this experiment, $A$ is generated as in \eqref{eq:A_randn_d} with $d:=4$, $n_i:=10$, for $i=1,\dots,4$, and $r_i':=4$, for $i=1,\dots, 3$. The rank of the lower-rank set was chosen as $r_i:=2$, for $i=1,\dots,3$. We compute the value
\begin{equation} \label{eq:inner_prod_angle_TT2}
\Bigg\langle \frac{\tilde{X}}{\|\tilde{X}\|}, \frac{A}{\|A\|} \Bigg\rangle,
\end{equation}
for the different random tensors $A$ and where $\tilde{X}$ is obtained with \Cref{alg:TT_rounding} because the value $\lVert \mathcal{P}_{\mathbb{R}_{(r_1, \dots, r_{d-1})}^{n_1 \times \cdots \times n_d}}A \rVert$ is not known. Remark that
\begin{equation*} 
\Bigg\langle \frac{\tilde{X}}{\|\tilde{X}\|}, \frac{A}{\lVert \mathcal{P}_{\mathbb{R}_{(r_1, \dots, r_{d-1})}^{n_1 \times \cdots \times n_d}}A \rVert} \Bigg\rangle \ge \Bigg\langle \frac{\tilde{X}}{\|\tilde{X}\|}, \frac{A}{\|A\|} \Bigg\rangle.
\end{equation*}
Thus, if \eqref{eq:inner_prod_angle_TT2} is larger than $\omega$ the angle condition is also satisfied.
In the left subfigure in \Cref{fig:test_angle_cond_d4_rAi4_ri2_TT-rounding_ni10_A_TTeMPS_randn}, the box plot for this experiment is shown and on the right, the value of \eqref{eq:inner_prod_angle_TT2} for each tensor $A$ and corresponding image of the approximate projection $\tilde{X}$. As can be seen, the value in \eqref{eq:inner_prod_angle_TT2} is always larger than $\omega$, which equals 0.5 and proves that the angle condition is satisfied for this experiment.

\begin{figure}[H]
\centering
\includegraphics[width= 0.75\linewidth]{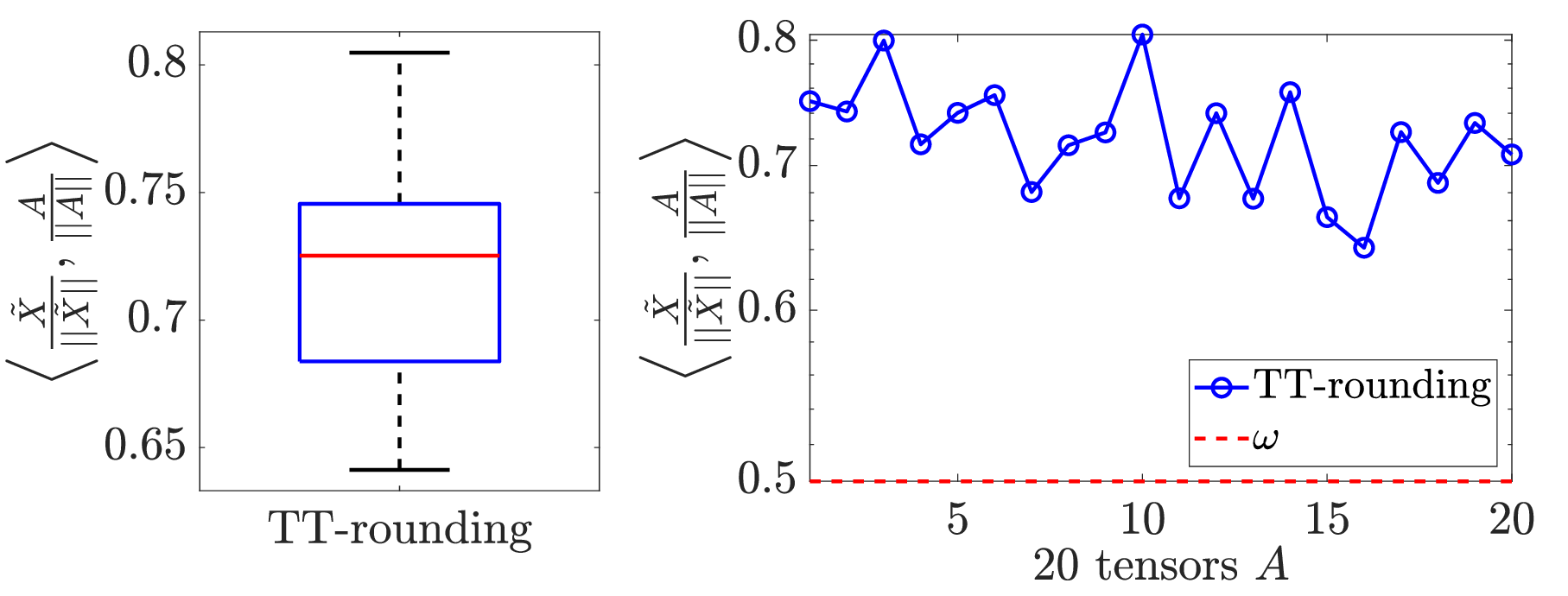}
\caption{Illustration of the angle condition of \Cref{alg:TT_rounding} for 20 randomly generated tensors $A$ as in \eqref{eq:A_randn_d} for $d:=4$, $n_i:=10$, for $i=1,\dots, 4$, $r_i':=4$ and $r_i:=2$, for $i=1,\dots,3$.}
\label{fig:test_angle_cond_d4_rAi4_ri2_TT-rounding_ni10_A_TTeMPS_randn}
\end{figure}

\section{Riemannian rank-adaptive method} \label{sec:RRAM_gen}
Now that all required elements are explained, the overall lay-out of the RRAM can be given and is shown in \Cref{alg:RRAM_ttcompl_dim_d}. The lay-out is inspired by the state-of-the-art methods for matrix and tensor completion \cite{gao2022riemannian,Zhou2016,Steinl_high_dim_TT_compl_2016}. The main contributions are the method to increase the rank using \Cref{alg:rank_incr_dim_d} on line 20 and the method to decrease the rank on line 14 with \Cref{alg:TT_rounding}.

\begin{algorithm}[H]
\setstretch{1.3}
\caption{$\mathrm{RRAM}^d_{\tilde{\mathcal{P}}}$ for higher-order tensor completion in the TT format.}
\label{alg:RRAM_ttcompl_dim_d}
\begin{algorithmic}[1] 
\Require
$\Omega, \Gamma, A_\Omega, A_\Gamma, X^{(0)}, \varepsilon_\Omega, \varepsilon_{\mathrm{R}} \in (0,1), \varepsilon_\Gamma \ge 0 , \Delta \in [0,1), j_{\max}, k_{\max} \in \mathbb{N}_0, r_{\max}, s_{\max} \in \mathbb{N}_0^{d-1}, \varepsilon > 0$
\State
$X_{\mathrm{new}} \gets X^{(0)}$;
\For{$k=1,\dots,k_{\max}$}
\State
$X^{(k)} \gets \mlin{Riem\_conjugate\_gradient} \left(X_{\mathrm{new}}, \varepsilon_\Omega,\varepsilon_{\mathrm{R}}, j_{\max} \right)$; \cite{Steinl_high_dim_TT_compl_2016}
\If{$\rho_{\Gamma} \left( X^{(k)} \right) > \varepsilon_\Gamma $}
\State $X^{(k)} \gets X^{(k-1)}$; 
\State \textbf{go to output}; 
\ElsIf{$\frac{\sqrt{2 f_\Omega (X^{(k)})}}{\lVert A_\Omega \rVert} < \varepsilon_\Omega$}
\State \textbf{go to output}; 
\EndIf
\If{$k< k_{\max}$}
\State $r \gets \rank X^{(k)}$; 
\State $r_\Delta \gets \rank X_\Delta$ \eqref{eq:num_rank_TTD};
\If{$\sum_{i=1}^{d-1} r_\Delta(i) < \sum_{i=1}^{d-1} r(i)$}
\State $X_\Delta \gets $ \Cref{alg:TT_rounding}$\left(X, r_\Delta \right)$;
\If{$f_\Omega \left( X_\Delta \right)< f_{\Omega} \left( X^{(k)} \right)$\Or $f_\Gamma \left( X_\Delta \right)< f_\Gamma \left( X^{(k)} \right)$}
\State $X_{\mathrm{new}} \gets X_\Delta$; 
\State \textbf{go to line 2 and continue for loop};
\EndIf 
\EndIf
\State
$X_{\mathrm{new}} \gets$ \Cref{alg:rank_incr_dim_d}$\left(\Omega,\Gamma,X^{(k)},r_{\max},s_{\max},\varepsilon \right)$;
\EndIf
\EndFor
\Ensure $X_k$
\end{algorithmic}
\end{algorithm}

The optimization on the smooth, fixed-rank manifold is displayed on line 3. We use a Riemannian conjugate gradient (CG) algorithm developed for tensor completion \cite{Steinl_high_dim_TT_compl_2016} for this purpose. This choice also allows us to make a good comparison with the state-of-the-art RRAM \cite{Steinl_high_dim_TT_compl_2016} in terms of rank adaptation. This RRAM and CG algorithm are both available in the \textsc{Manopt} toolbox \cite{manopt}.

\subsection{Overfitting}
As already briefly discussed in the introduction, overfitting is a common difficulty for the tensor completion problem \eqref{eq:min_completion_TT}. Overfitting occurs if $k$ is too high in at least one mode. In this case, $X$ can still approximate $A_{\Omega}$ well (small $f_\Omega(X)$), but $f(X)$ is large, assuming $A$ is known. That is why usually in the literature a test set $\Gamma$ is included \cite{Steinl_high_dim_TT_compl_2016}. The corresponding error function is $f_{\Gamma}(X)$. If 
\begin{equation} \label{eq:rho_Gamma}
\rho_\Gamma \left( X^{(k)} \right):= \frac{f_\Gamma \left( X^{(k)}\right)- f_\Gamma \left(X^{(k-1)}\right)}{f_\Gamma \left( X^{(k-1)}\right)}
\end{equation}
is larger than some positive constant $\varepsilon_{\Gamma}$ at iteration $k$, the RRAM is stopped and $X^{(k-1)}$ is considered as the best solution. The set $\Gamma$ is typically smaller than $\Omega$, e.g., $\lvert \Gamma \rvert = 0.25 \cdot \lvert \Omega \rvert$.

\subsection{Stopping criterion}

The inner Riemannian CG algorithm runs until one of the following conditions is satisfied:
\begin{itemize}\setlength\itemsep{1em} 
\item a maximal number of inner iterations $j_{\max}$ is reached,
\item the norm of the Riemannian gradient is smaller than $\varepsilon_{\mathrm{R}}$,
\item $\sqrt{2 f_\Omega (X^{(j)})}/\lVert A_\Omega \rVert$ is smaller than $\varepsilon_{\Omega}$,
\item $\frac{\sqrt{2 f_\Omega (X^{(j)})}-\sqrt{2 f_\Omega (X^{(j-1)})}}{\sqrt{2 f_\Omega (X^{(j)})}}$ is smaller than a certain tolerance, which is by default set to $10^{-8}$. We did not change this value.
\end{itemize} 
On the other hand, \Cref{alg:RRAM_ttcompl_dim_d} runs until
\begin{itemize}\setlength\itemsep{1em}
\item the maximal number of outer iterations $k_{\max}$ is reached,
\item $\sqrt{2 f_\Omega (X^{(k)})}/\lVert A_\Omega \rVert$ is smaller than $\varepsilon_{\Omega}$,
\item the relative error on the test set $\rho_\Gamma (X^{(k)})$ in \eqref{eq:rho_Gamma} has increased more than $\varepsilon_{\Gamma}$. 
\end{itemize} 

A convergence analysis is out of the scope of this paper but might be interesting to perform in future work.

\subsection{Experiments}

In this section, we compare the results that we obtain with the algorithm proposed in \Cref{alg:RRAM_ttcompl_dim_d} with the one from \cite{Steinl_high_dim_TT_compl_2016}. Both algorithms are denoted by $\mathrm{RRAM}^d_{\tilde{\mathcal{P}}}$ and $\mathrm{RRAM}_{\mathrm{randn}}$ respectively. In the first subsection $A$ is generated as in \eqref{eq:A_randn_d} of known low TT-rank and in the second subsection, $A$ is obtained by evaluating an exponential function in four variables in equidistant points.

\subsubsection{Synthetic data}
In this subsection, we generate $A$ and $X_0$ as in \eqref{eq:A_randn_d} for different values of $d$, $n$, and $r'$. The following input parameters for \Cref{alg:RRAM_ttcompl_dim_d} were chosen:
\begin{align*}
\varepsilon_{\Omega} &:= 10^{-8},& \varepsilon &:= 10^{-10}, & j_{\max} &:= 15,& s_{\max}(i)&:=8,\\
\varepsilon_{\mathrm{R}} &:=10^{-8},& \varepsilon_{\Gamma} &:=1,   & k_{\max} &:=15, & r_{\max}(i)&:=10, & \Delta &:=0.8,
\end{align*}
for $i=1,\dots,d-1$. Furthermore, $\mathrm{RRAM}_{\mathrm{randn}}$ needs as input an upper bound on the TT-rank. However, there is no possibility to use different values for each mode. Thus, we gave in all experiments in this section $\max_i(r_i')$ as input. This is in practice of course not possible when the rank is not known.

In the first experiment, we choose $d:=4$, $n_i:=20$, for $i=1,\dots,4$, $r_i':=4$, for $i=1,\dots,3$, and $\rho_\Omega := 0.1$. The samples are generated in {\Matlab} using \mlin{randperm} over all indices. The results are shown in \Cref{fig:comp_RRAM_rAi4_ni20_rho_omega_0_1_rng1_d4}. In the left subfigure, the convergence of the relative value of the cost function (w.r.t.\ the norm of $A_\Omega$) is shown for both algorithms. The x-axis here indicates the cumulative number of inner iterations of the CG algorithm. The middle subfigure shows the convergence of the relative value of the test function. Now the x-axis shows the time in seconds. The rightmost column of subfigures shows the evolution of the TT-rank. As explained in the introduction, the rank increase of $\mathrm{RRAM}_{\mathrm{randn}}$ is fixed and increases always by one. This algorithm thus needs 10 outer iterations or approximately five seconds to reach the rank of $A$ and obtain a numerically exact approximation. On the other hand, the proposed RRAM is able to detect the TT-rank of $A$ after the first outer iteration and converges to a numerically exact approximation in less than two seconds.

\begin{figure}[H]
\centering
\includegraphics[width= 0.8\linewidth]{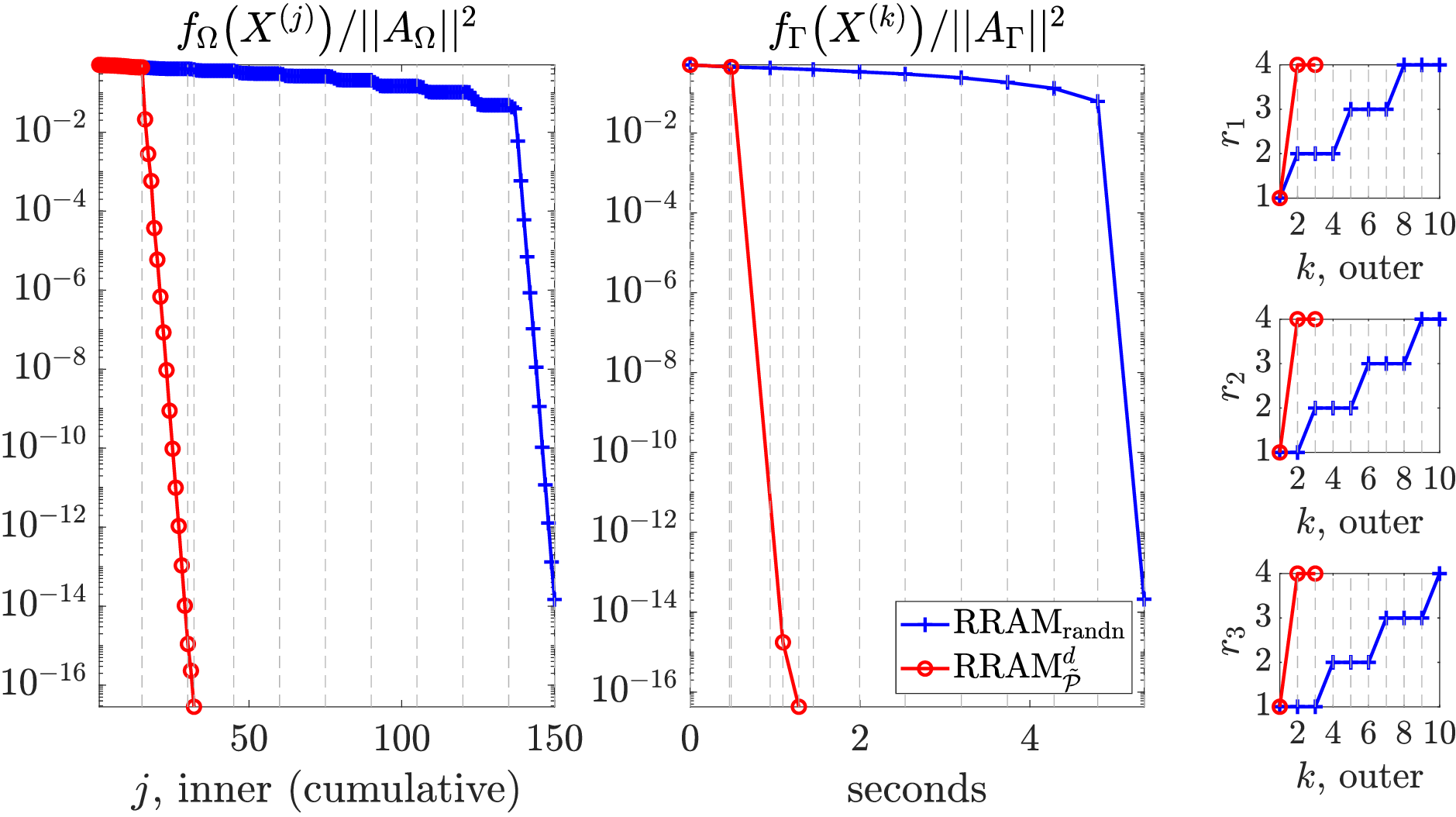}
\caption{Comparison of $\mathrm{RRAM}^d_{\tilde{\mathcal{P}}}$ proposed in \Cref{alg:RRAM_ttcompl_dim_d} and $\mathrm{RRAM}_{\mathrm{randn}}$ \cite{Steinl_high_dim_TT_compl_2016} for $d:=4$, $n_i:=20$, for $i=1,\dots,4$, $\rho_\Omega := 0.1$, $r_i':=4$, for $i=1,\dots,3$, and $A$ generated as in \eqref{eq:A_randn_d}.}
\label{fig:comp_RRAM_rAi4_ni20_rho_omega_0_1_rng1_d4}
\end{figure}

In a second experiment, we increase the order to $d:=5$. The other parameters remain the same. The results are shown in \Cref{fig:comp_RRAMs_rAi4_ni20_rhoOmega_0_1prodn_rng1_maxit15_d5}. As can be seen, due to the higher order $\mathrm{RRAM}_{\mathrm{randn}}$ needs approximately 100 seconds, 13 outer, and 200 inner iterations to converge to a numerically exact approximation whereas $\mathrm{RRAM}^d_{\tilde{\mathcal{P}}}$ is again able to detect the TT-rank of $A$ after one outer iteration and converges to a numerically exact approximation in approximately $10$ seconds and 25 inner iterations.

\begin{figure}[H]
\centering
\includegraphics[width= 0.8\linewidth]{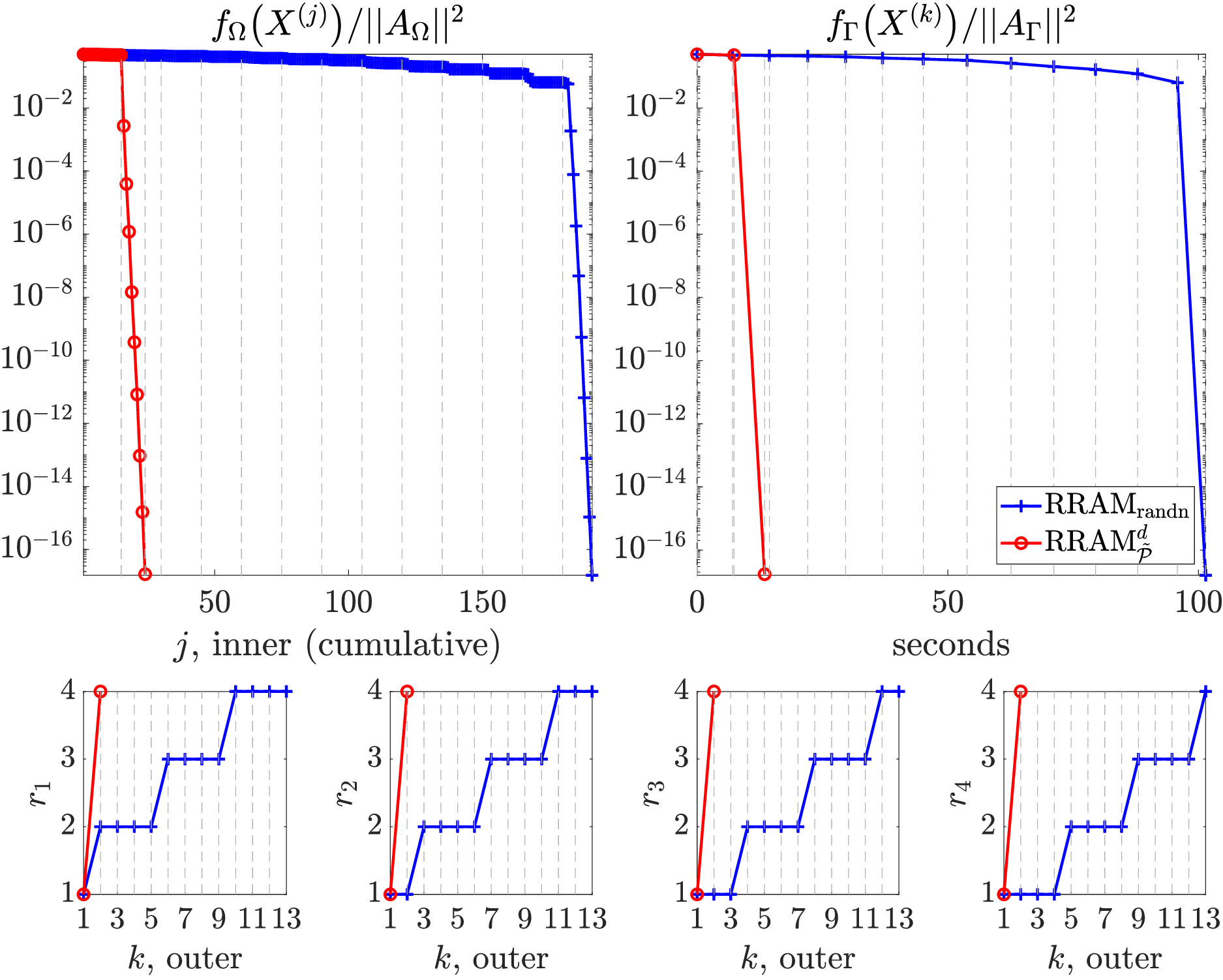}
\caption{Comparison of $\mathrm{RRAM}^d_{\tilde{\mathcal{P}}}$ proposed in \Cref{alg:RRAM_ttcompl_dim_d} and $\mathrm{RRAM}_{\mathrm{randn}}$ \cite{Steinl_high_dim_TT_compl_2016} for $d:=5$, $n_i:=20$, for $i=1,\dots,5$, $\rho_\Omega := 0.1$, $r_i':=4$, for $i=1,\dots,4$, and $A$ generated as in \eqref{eq:A_randn_d}.}
\label{fig:comp_RRAMs_rAi4_ni20_rhoOmega_0_1prodn_rng1_maxit15_d5}
\end{figure}

In a third experiment, noise is added to the tensor from the previous experiment as in \eqref{eq:A_noise} of size $10^{-1}$. The results are shown in \Cref{fig:comp_RRAMs_rAi4_ni20_rhoOmega_0_1prodn_rng1_maxit15_d5_noise_10-1}. Because of the noise, the most accurate value of the relative cost function that is obtained is approximately $2\cdot 10^{-5}$. After the first outer iteration, the estimated TT-rank is slightly too high but after the second outer iteration, the rank is reduced to the correct value. The tolerance on the cost function $\varepsilon_{\Omega}$ was lowered to $10^{-3}$. This value is however not reached and \Cref{alg:RRAM_ttcompl_dim_d} stops because the tolerance on the gradient, equal to $10^{-8}$, is reached. Still $\mathrm{RRAM}^d_{\tilde{\mathcal{P}}}$ obtains a solution of the same accuracy faster than $\mathrm{RRAM}_{\mathrm{randn}}$, both in number of iterations as computation time.

\begin{figure}[H]
\centering
\includegraphics[width= 0.8\linewidth]{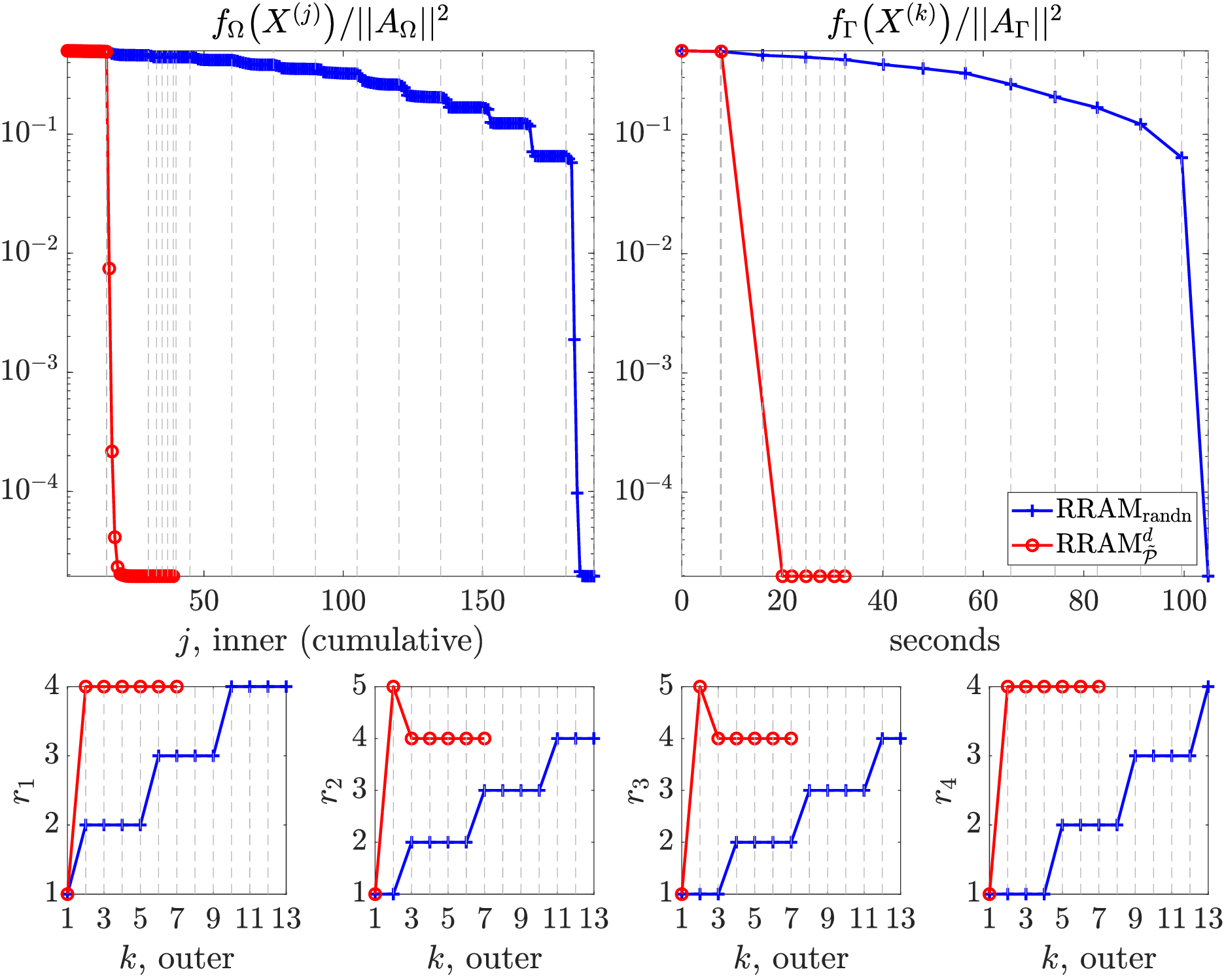}
\caption{Comparison of $\mathrm{RRAM}^d_{\tilde{\mathcal{P}}}$ proposed in \Cref{alg:RRAM_ttcompl_dim_d} and $\mathrm{RRAM}_{\mathrm{randn}}$ \cite{Steinl_high_dim_TT_compl_2016} for $d:=5$, $n_i:=20$, for $i=1,\dots,5$, $\rho_\Omega := 0.1$, $r_i':=4$, for $i=1,\dots,4$, and $A$ generated as in \eqref{eq:A_randn_d} with noise added of size $10^{-1}$ as in \eqref{eq:A_noise}.}
\label{fig:comp_RRAMs_rAi4_ni20_rhoOmega_0_1prodn_rng1_maxit15_d5_noise_10-1}
\end{figure}

In a last example, the order is further increased to $d:=6$ and the TT-rank of $A$ is chosen as $r:= [2,4,5,4,2]$. The size of $A$ in each dimension is lowered to 15. The results are shown in \Cref{fig:comp_RRAMs_rA24542_ni15_rhoOmega_0_1prodn_rng1_maxit15_d6}. Again, the proposed algorithm is significantly faster than the one from \cite{Steinl_high_dim_TT_compl_2016}, both in number of iterations and computation time. Furthermore, the rank is estimated correctly again after only one outer iteration.

\begin{figure}[H]
\centering
\includegraphics[width= 0.9\linewidth]{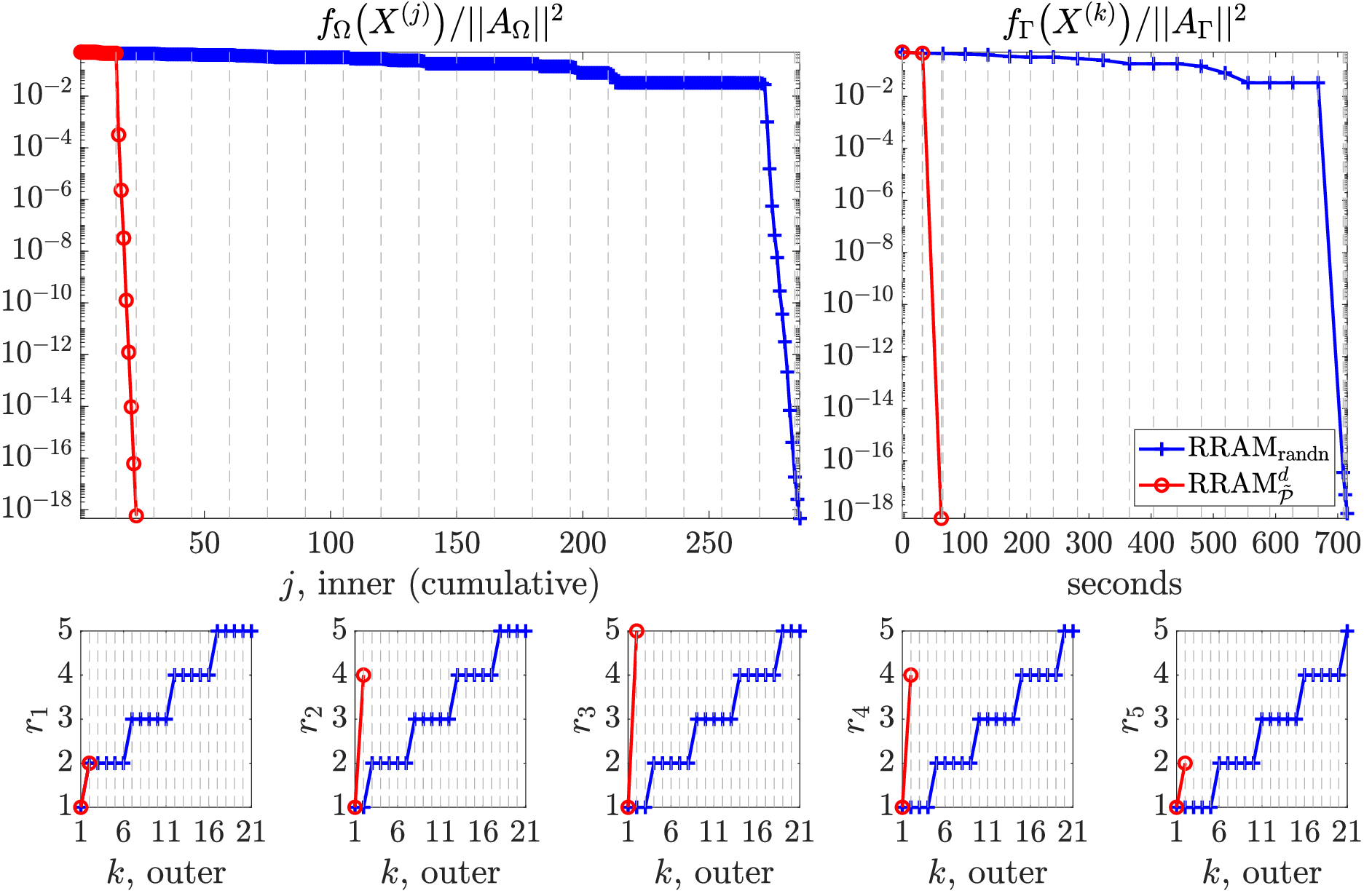}
\caption{Comparison of $\mathrm{RRAM}^d_{\tilde{\mathcal{P}}}$ proposed in \Cref{alg:RRAM_ttcompl_dim_d} and $\mathrm{RRAM}_{\mathrm{randn}}$ \cite{Steinl_high_dim_TT_compl_2016} for $d:=6$, $n_i:=15$, for $i=1,\dots,6$, $\rho_\Omega := 0.1$, $r':=[2,4,5,4,2]$, and $A$ generated as in \eqref{eq:A_randn_d}.}
\label{fig:comp_RRAMs_rA24542_ni15_rhoOmega_0_1prodn_rng1_maxit15_d6}
\end{figure}

\subsubsection{Function interpolation}
In this section, we reproduce one of the experiments in \cite{Steinl_high_dim_TT_compl_2016}. The following real function in four variables $x_1,x_2,x_3,x_4 \in [0,1]$ is considered:
\begin{equation*}
f(x_1,x_2,x_3,x_4) := \mathrm{exp}\left(-\left\lVert \begin{matrix}
x_1 & x_2 & x_3 & x_4
\end{matrix} \right\rVert\right).
\end{equation*}
The function is sampled in 20 uniformly distributed points in each variable interval to obtain the following tensor:
\begin{equation} \label{eq:A_exp}
A(i_1,i_2,i_3,i_4) := \mathrm{exp}\left(-\left\lVert \begin{matrix}
\frac{i_1-1}{19} & \frac{i_2-1}{19} & \frac{i_3-1}{19} & \frac{i_4-1}{19}
\end{matrix} \right\rVert\right),
\end{equation}
for $i_1,i_2,i_3,i_4=1,\dots,20$. The input parameters for \Cref{alg:RRAM_ttcompl_dim_d} were the same as in the previous subsection, except for $s_{\max}(i):=5$ and $r_{\max}(i):=5$, for $i=1,\dots,d-1$.
The maximal rank for $\mathrm{RRAM}_{\mathrm{randn}}$ was also set to 5 in each mode and $\rho_\Omega := 0.1$, both as in the experiment in \cite{Steinl_high_dim_TT_compl_2016}. The results are shown in \Cref{fig:comp_RRAMs_A_exp_ni20_rhoOmega_0_1prodn_rng1_maxit15_d4}. As can be seen, \Cref{alg:RRAM_ttcompl_dim_d} estimates the TT-rank after each outer iteration as 1 in each mode because the tensor samples an exponential function and consequently also the singular values of the unfoldings of $A$ are decaying exponentially as can be seen in \Cref{fig:svd_unfoldings_Aexp_d4}. The most accurate relative cost function that both methods are able to obtain is approximately $10^{-9}$. This is as expected as $A$ is only approximately of low rank. However, $\mathrm{RRAM}^d_{\tilde{\mathcal{P}}}$ is still able to reach the same accuracy in cost and test function faster both in terms of iterations as computation time.

\begin{figure}[H]
\centering
\includegraphics[width= 0.8\textwidth]{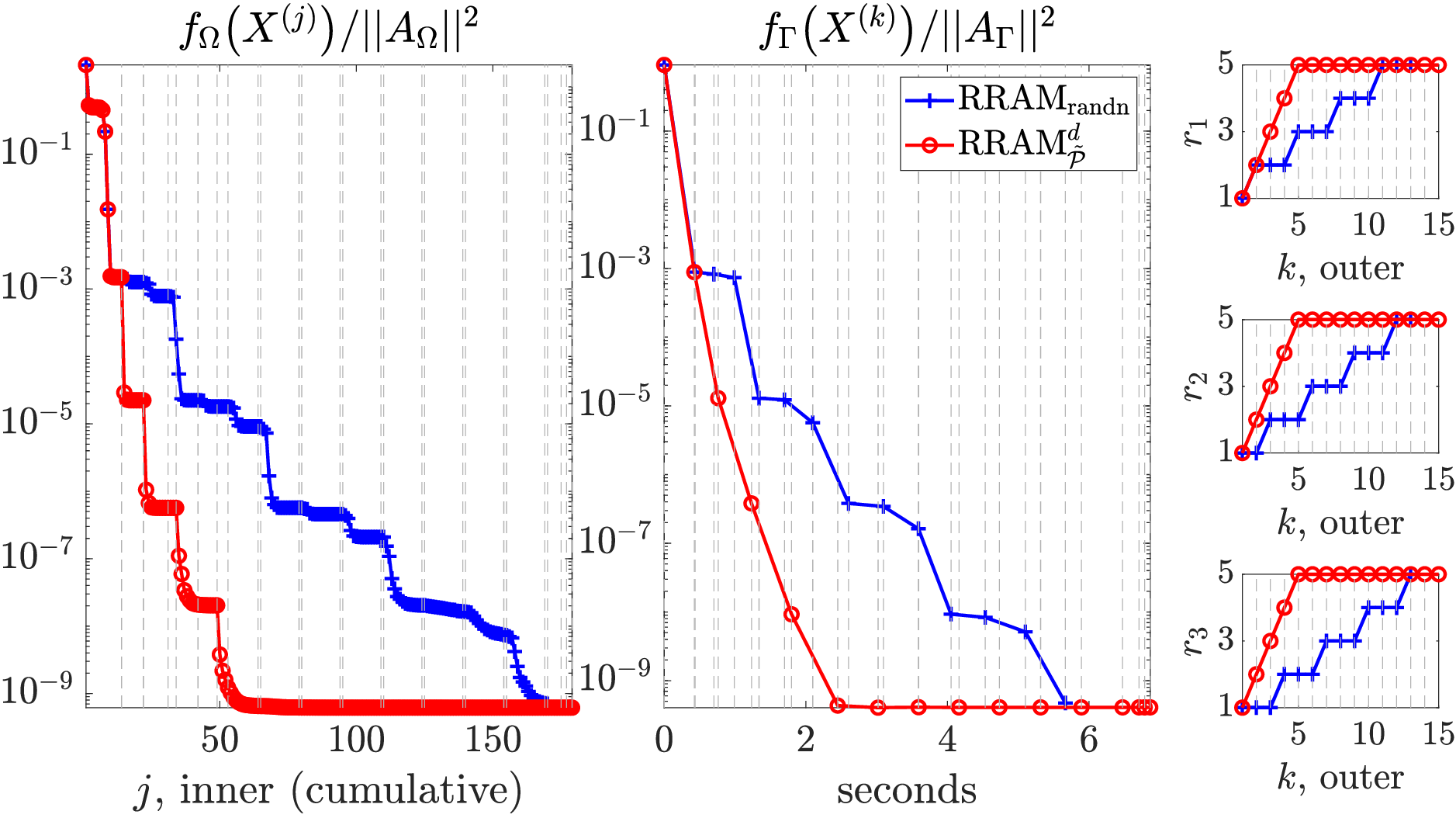}
\caption{Comparison of $\mathrm{RRAM}^d_{\tilde{\mathcal{P}}}$ proposed in \Cref{alg:RRAM_ttcompl_dim_d} and $\mathrm{RRAM}_{\mathrm{randn}}$ \cite{Steinl_high_dim_TT_compl_2016} for $d:=4$, $n_i:=20$, for $i=1,\dots,4$, $\rho_\Omega := 0.1$, and $A$ generated as in \eqref{eq:A_exp}.}
\label{fig:comp_RRAMs_A_exp_ni20_rhoOmega_0_1prodn_rng1_maxit15_d4}
\end{figure}

\begin{figure}[H]
\centering
\includegraphics[width= 0.6\textwidth]{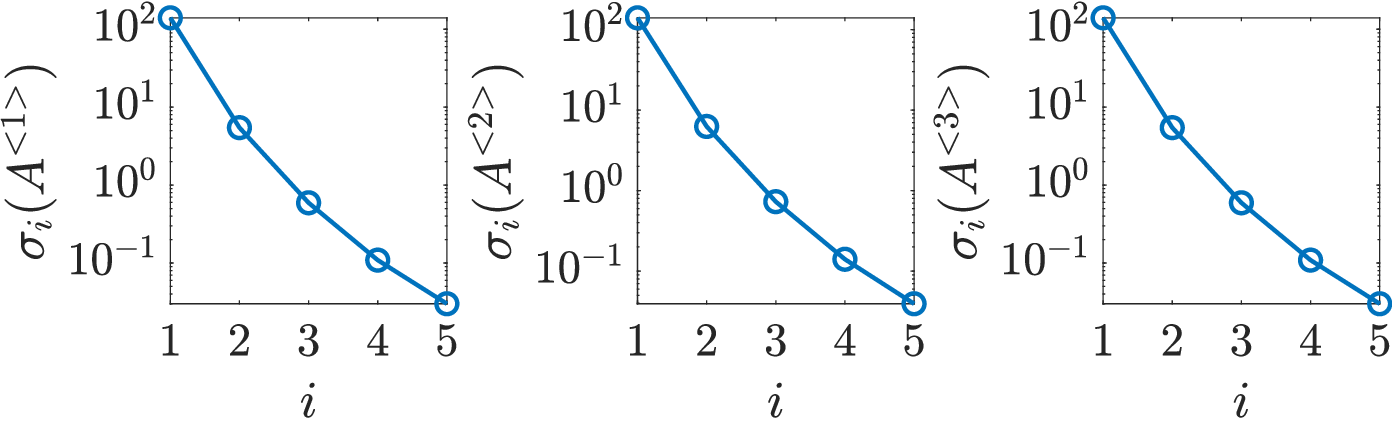}
\caption{The first 5 singular values of the unfoldings of $A$ generated as in \eqref{eq:A_exp}.}
\label{fig:svd_unfoldings_Aexp_d4}
\end{figure}

\section{Conclusions}\label{sec5}

A Riemannian rank-adaptive method is proposed for the tensor completion problem in the low-rank tensor-train format which improves the state-of-the-art RRAM by including a method to increase the rank based on successive projections of the negative gradient onto subcones in the normal part of the tangent cone to the variety of bounded tensor-train rank tensors. Furthermore, a rank estimation method is included to estimate a good value for the amount of rank increase. When the tensor to complete is of exact low-rank, the method is able to retrieve this rank. Additionally, when the algorithm converges to an element of a lower-rank set, the rank is reduced based on the TT-rounding algorithm \cite{Oseledets2011}, which can be considered as an approximate projection onto the lower-rank set and is proven to satisfy a certain angle condition to ensure that the image is sufficiently close to one of the true projection. Several numerical experiments with synthetic tensors and a tensor obtained from function evaluations were given. In all experiments, the proposed RRAM was able to recover the low-rank tensor or obtain a good approximation of the full rank tensor faster than the state-of-the-art RRAM, both in terms of iterations and computation time.
\bibliographystyle{acm}
\bibliography{references}

\begin{thebibliography}{10}

\bibitem{manopt}
{\sc Boumal, N., Mishra, B., Absil, P.-A., and Sepulchre, R.}
\newblock {M}anopt, a {M}atlab toolbox for optimization on manifolds.
\newblock {\em Journal of Machine Learning Research 15}, 42 (2014), 1455--1459.

\bibitem{budzinskiy2021tensor}
{\sc Budzinskiy, S., and Zamarashkin, N.}
\newblock Tensor train completion: local recovery guarantees via {R}iemannian
  optimization.
\newblock {\em Numerical Linear Algebra with Applications 30}, 6 (2023), e2520.

\bibitem{cai2023tensor}
{\sc Cai, J.-F., Huang, W., Wang, H., and Wei, K.}
\newblock Tensor completion via tensor train based low-rank quotient geometry
  under a preconditioned metric.
\newblock {\em arXiv preprint\/} (2022).

\bibitem{gao2022riemannian}
{\sc Gao, B., and Absil, P.-A.}
\newblock A {R}iemannian rank-adaptive method for low-rank matrix completion.
\newblock {\em Computational Optimization and Applications 81\/} (2022),
  67--90.

\bibitem{Hackbusch}
{\sc Hackbusch, W.}
\newblock {\em Tensor Spaces and Numerical Tensor Calculus}, 2nd~ed., vol.~56
  of {\em Springer Series in Computational Mathematics}.
\newblock Springer Cham, 2019.

\bibitem{hillar2013most}
{\sc Hillar, C.~J., and Lim, L.-H.}
\newblock Most tensor problems are {NP}-hard.
\newblock {\em Journal of the ACM 60}, 6 (nov 2013).

\bibitem{Holtz2012}
{\sc Holtz, S., Rohwedder, T., and Schneider, R.}
\newblock The alternating linear scheme for tensor train optimization in the
  tensor train format.
\newblock {\em SIAM Journal on Scientific Computing 34}, 2 (2012), A683--A713.

\bibitem{Holtz_manifolds_2012}
{\sc Holtz, S., Rohwedder, T., and Schneider, R.}
\newblock On manifolds of tensors of fixed {TT}-rank.
\newblock {\em Numerische Mathematik 120}, 4 (2012), 701--731.

\bibitem{TTcompletion_precond_kasai_2016}
{\sc Kasai, H., and Mishra, B.}
\newblock Low-rank tensor completion: a {R}iemannian manifold preconditioning
  approach.
\newblock In {\em Proceedings of The 33rd International Conference on Machine
  Learning\/} (New York, New York, USA, 20--22 Jun 2016), M.~F. Balcan and
  K.~Q. Weinberger, Eds., vol.~48 of {\em Proceedings of Machine Learning
  Research}, PMLR, pp.~1012--1021.

\bibitem{tensor_completioni_reg_TT_2020}
{\sc Ko, C.-Y., Batselier, K., Daniel, L., Yu, W., and Wong, N.}
\newblock Fast and accurate tensor completion with total variation regularized
  tensor trains.
\newblock {\em IEEE Transactions on Image Processing 29\/} (2020), 6918--6931.

\bibitem{LR_tensor_methods_Steinlechner_2014}
{\sc Kressner, D., Steinlechner, M., and Uschmajew, A.}
\newblock Low-rank tensor methods with subspace correction for symmetric
  eigenvalue problems.
\newblock {\em SIAM J. Sci. Comput. 36}, 5 (2014), A2346--A2368.

\bibitem{kressner2014low}
{\sc Kressner, D., Steinlechner, M., and Vandereycken, B.}
\newblock Low-rank tensor completion by {R}iemannian optimization.
\newblock {\em BIT Numerical Mathematics 54\/} (2014), 447--468.

\bibitem{kutschan2018}
{\sc Kutschan, B.}
\newblock Tangent cones to tensor train varieties.
\newblock {\em Linear Algebra and its Applications 544\/} (2018), 370--390.

\bibitem{LevinKileelBoumal2022}
{\sc Levin, E., Kileel, J., and Boumal, N.}
\newblock Finding stationary points on bounded-rank matrices: A geometric
  hurdle and a smooth remedy.
\newblock {\em Mathematical Programming\/} (2022).

\bibitem{Lubich_TT_time_int_2015}
{\sc Lubich, C., Oseledets, I.~V., and Vandereycken, B.}
\newblock Time integration of tensor trains.
\newblock {\em SIAM Journal on Numerical Analysis 53}, 2 (2015), 917--941.

\bibitem{Oseledets2011}
{\sc Oseledets, I.~V.}
\newblock Tensor-train decomposition.
\newblock {\em SIAM Journal on Scientific Computing 33}, 5 (2011), 2295--2317.

\bibitem{psenka2020second}
{\sc Psenka, M., and Boumal, N.}
\newblock Second-order optimization for tensors with fixed tensor-train rank.
\newblock {\em arXiv preprint\/} (2020).

\bibitem{schneider2015convergence}
{\sc Schneider, R., and Uschmajew, A.}
\newblock Convergence results for projected line-search methods on varieties of
  low-rank matrices via {{\L}}ojasiewicz inequality.
\newblock {\em SIAM Journal on Optimization 25}, 1 (2015), 622--646.

\bibitem{Steinl_high_dim_TT_compl_2016}
{\sc Steinlechner, M.}
\newblock {R}iemannian optimization for high-dimensional tensor completion.
\newblock {\em SIAM Journal on Scientific Computing 38}, 5 (2016), S461--S484.

\bibitem{Steinlechner_thesis2016}
{\sc Steinlechner, M.~M.}
\newblock {\em {R}iemannian Optimization for Solving High-Dimensional Problems
  with Low-Rank Tensor Structure}.
\newblock PhD thesis, MATHICSE, Lausanne, 2016.

\bibitem{rank_est_3D}
{\sc Vermeylen, C., Olikier, G., Absil, P.-A., and Van~Barel, M.}
\newblock Rank estimation for third-order tensor completion in the tensor-train
  format.
\newblock In {\em 31st European Signal Processing Conference (EUSIPCO)\/}
  (2023), pp.~965--969.

\bibitem{approx_proj_3D}
{\sc Vermeylen, C., Olikier, G., and Van~Barel, M.}
\newblock An approximate projection onto the tangent cone to the variety of
  third-order tensors of bounded tensor-train rank.
\newblock In {\em Geometric Science of Information.\/} (Cham, 2023), F.~Nielsen
  and F.~Barbaresco, Eds., Springer Nature Switzerland, pp.~484--493.

\bibitem{Zhou2016}
{\sc Zhou, G., Huang, W., Gallivan, K.~A., {Van Dooren}, P., and Absil, P.-A.}
\newblock A {R}iemannian rank-adaptive method for low-rank optimization.
\newblock {\em Neurocomputing 192\/} (2016), 72--80.

\end{thebibliography}
\section*{Data availability}

Our implementation of the RRAM -- \Cref{alg:RRAM_ttcompl_dim_d} -- and the scripts to regenerate the experiments are publicly available \footnote{\url{https://github.com/CharlotteVermeylen/RRAM_TT_completion}}.
\appendix

\section{Orthogonal projections\label{app1.1a}}

Projections onto vector spaces are frequently used in this work. More specifically, the proof of the angle condition in \Cref{thm:angle_cond_rank_red_gen} relies on the following basic result.

\begin{lemma}
Let $A \in \mathbb{R}^{n \times m}$ have rank $r$. If $\hat{A} = \hat{U} \hat{S} \hat{V}^\top$ is a truncated SVD of rank $s$ of $A$, with $s<r$, then, for all $U \in \Stiefel(s, n)$ and all $V \in \Stiefel(s, m)$,
\begin{align}
\label{eq:ineqs_svd_trunc1}
\left\|P_{\hat{U}} A \right\| \ge \left\|P_U A \right\|,&&
\left\|P_{\hat{U}} A \right\|^2 \ge \frac{s}{r} \|A\|^2,\\ \label{eq:ineqs_svd_trunc2}
\left\lVert A P_{\hat{V}} \right\rVert \ge \left\|A P_V \right\|,&&
\left\|A P_{\hat{V}} \right\|^2 \ge \frac{s}{r} \|A\|^2.
\end{align}
\end{lemma}

\begin{proof}
By the Eckart--Young theorem, $\hat{A}$ is a projection of $A$ onto 
$$\mathbb{R}_{\le s}^{n \times m} := \left\{X \in \mathbb{R}^{n \times m} \bigm\vert \rank(X) \le s \right\}.$$ 
Thus, since $\mathbb{R}_{\le s}^{n \times m}$ is a closed cone, \eqref{eq:nec_cond_approx_proj} holds. Moreover, since $\hat{S} \hat{V}^\top = \hat{U}^\top A$ and thus $\hat{A} = \hat{U} \hat{U}^\top A = P_{\hat{U}} A$, it holds that
\begin{equation*}
\lVert P_{\hat{U}} A \rVert^2 = \max \left\{\lVert A_1 \rVert^2 \bigm\vert A_1 \in \mathbb{R}_{\le s}^{n \times m},\, \langle A_1, A\rangle = \lVert A_1 \rVert^2\right\}.
\end{equation*}
Furthermore, for all $U \in \Stiefel(s, n)$,
$\left\langle P_{{U}}A, A \right\rangle = \left\langle P_{{U}} A, P_{{U}} A + P_{{U}}^\perp A \right\rangle = \left\lVert P_{{U}} A \right\rVert^2.$
Hence,
\begin{equation*}
\left\{ P_{U} A \mid U \in \Stiefel(s, n) \right\} \subseteq \left\{\lVert A_1 \rVert^2 \bigm\vert A_1 \in \mathbb{R}_{\le s}^{n \times m},\, \langle A_1, A\rangle = \lVert A_1 \rVert^2\right\}.
\end{equation*}
Thus,
$\lVert P_{\hat{U}} A \rVert^2 = \underset{U \in \Stiefel(s, n)}{\max} \lVert P_{U} A \rVert^2.$
The left inequality in \eqref{eq:ineqs_svd_trunc1} follows, and the one in \eqref{eq:ineqs_svd_trunc2} can be obtained similarly.

By orthogonal invariance of the Frobenius norm and by definition of $\hat{A}$,
\begin{align*}
\lVert A \rVert^2
= \sum_{i=1}^r \sigma_i^2,&&
\lVert \hat{A} \rVert^2
= \sum_{i=1}^s \sigma_i^2,
\end{align*}
where $\sigma_1, \dots, \sigma_r$ are the singular values of $A$ in decreasing order.
Moreover, either $\sigma_s^2 \ge \frac{1}{r} \sum_{i=1}^r \sigma_i^2$ or $\sigma_s^2 < \frac{1}{r} \sum_{i=1}^r \sigma_i^2$.
In the first case, we have
\begin{equation*}
\lVert \hat{A} \rVert^2
= \sum_{i=1}^s \sigma_i^2
\ge s \sigma_s^2
\ge s \frac{\sum_{i=1}^r \sigma_i^2}{r}
= \frac{s}{r} \lVert A \rVert^2.
\end{equation*}
In the second case, we have
\begin{align*}
\lVert \hat{A} \rVert^2
= \sum_{i=1}^r \sigma_i^2 - \sum_{i=s+1}^r \sigma_i^2
&\ge \lVert A \rVert^2 - (r-s) \sigma_s^2 > \lVert A \rVert^2 - (r-s) \frac{\sum_{i=1}^r \sigma_i^2}{r}
= \frac{s}{r} \lVert A \rVert^2.
\end{align*}
Thus, in both cases, the second inequality in \eqref{eq:ineqs_svd_trunc1} holds. The second inequality in \eqref{eq:ineqs_svd_trunc2} can be obtained in a similar way.
\end{proof}

\end{document}